\documentclass[11pt]{amsart}

\usepackage[top=1in, bottom=1.5in, left=1.6in, right=1.7in]{geometry}

\usepackage[breaklinks]{hyperref}
\hypersetup{%
  pdftitle={\@title},
  pdfauthor={\@author},
  unicode=true,
  colorlinks=true,
  linkcolor=blue,
  citecolor=blue,
  filecolor=blue,
  urlcolor=blue
}

\usepackage[utf8]{inputenc}
\usepackage[T1]{fontenc}
\usepackage[english]{babel}

\usepackage[usenames,dvipsnames]{xcolor}

\usepackage[authoryear,round]{natbib}

\usepackage{amsmath,amsfonts,amssymb,amsthm,amscd}

\newtheorem{theorem}{Theorem}
\newtheorem{proposition}{Proposition}
\newtheorem{lemma}{Lemma}
\newtheorem{corollary}{Corollary}

\usepackage{nameref}
\usepackage[nameinlink]{cleveref}
\crefname{equation}{equation}{equations}
\crefname{assumption}{assumption}{assumptions}
\creflabelformat{enumi}{(#2#1#3)}

\usepackage{mathtools}

\providecommand\given{}
\newcommand\SetSymbol[1][]{
  \nonscript\,#1:\nonscript\,\mathopen{}\allowbreak}
\DeclarePairedDelimiterX\Set[1]{\lbrace}{\rbrace}%
{ \renewcommand\given{\SetSymbol[]} #1 }

\allowdisplaybreaks

\usepackage{amsaddr}

\usepackage{stmaryrd}
\usepackage{mathrsfs}
\usepackage{bbm}
\DeclareMathOperator{\Ind}{\mathbbm{1}}
\usepackage{booktabs}

\newcommand{\iid}{\overset{\mathrm{i.i.d}}{\sim}}

\title[Posterior concentration rates for mixtures of Gaussians]{Tails
assumptions and posterior concentration rates for mixtures of Gaussians}
\author{Zacharie Naulet \and Judith Rousseau}
\address{CEREMADE, Université Paris-Dauphine, France}
\email{zacharie.naulet@dauphine.eu; rousseau@ceremade.dauphine.fr}
\date{\today}

\begin{document}

\begin{abstract}
  Nowadays in density estimation, posterior rates of convergence for location and
location-scale mixtures of Gaussians are only known under light-tail
assumptions; with better rates achieved by location mixtures. It is conjectured,
but not proved, that the situation should be reversed under heavy tails
assumptions. The conjecture is based on the feeling that there is no need to
achieve a good order of approximation in regions with few data (say, in the
tails), favoring location-scale mixtures which allow for spatially varying order
of approximation. Here we test the previous argument on the Gaussian errors mean
regression model with random design, for which the light tail assumption is not
required for proofs. Although we cannot invalidate the conjecture due to the
lack of lower bound, we find that even with heavy tails assumptions,
location-scale mixtures apparently perform always worst than location
mixtures. However, the proofs suggest to introduce \textit{hybrid}
location-scale mixtures that are find to outperform both location and
location-scale mixtures, whatever the nature of the tails. Finally, we show that
all tails assumptions can be released at the price of making the prior
distribution covariate dependent.

\end{abstract}

\maketitle

\section{Introduction}
\label{sec:introduction}

Nonparametric mixture models are highly popular in the Bayesian nonparametric
literature, due to both their reknown flexibility and relative easiness of
implementation, see \citet{hjort:holmes:mueller:2010} for a review. They have
been used in particular for density estimation, clustering and classification
and recently nonparametric mixture models have also been proposed in nonlinear
regression models, see for instance \citet{dejonge:vzanten:09,
  WolpertClydeTu2011, NauletBarat2015}.

There is now a large literature on posterior concentration rates for
nonparametric mixture models, initiated by \citet{Ghosal2001,Ghosal2007} and
improved by \citet{kruijer:rousseau:vdv:10,ShenTokdarGhosal2013,scricciolo:12}
in the context of location mixtures of Gaussian distributions and studied by
\citet{Canale2013} in the context of location-scale Gaussian distributions and
\citet{dejonge:vzanten:09} in the case of location mixture models for nonlinear
regression.

Location mixture of Gaussian densities can be writen as
\begin{equation}
  \label{eq:location}
  f_{\sigma, G} (x) = \int_{\mathbb R}  \varphi_\sigma(x-\mu) dG(\mu),
\end{equation}
while location-scale mixtures have the form
\begin{equation}
  \label{eq:locationscale}
  f_{G} (x) = \int_{\mathbb R\times \mathbb R^+}  \varphi_\sigma(x-\mu) dG(\mu,
  \sigma).
\end{equation}

These models are used in the Bayesian nonparametric literature to model smooth
curves, typically probability densities, by putting a prior on the mixing
distribution $G$ (and on $\sigma$ for location mixtures
\eqref{eq:location}). The most popular prior distributions on $G$ are either
finite with unknown number of components, as in \citet{kruijer:rousseau:vdv:10}
and the reknown Dirichlet Process (\citet{Ferguson1973}) or some of its
extensions. In both cases $G$ is discrete almost surely.

In \citet{kruijer:rousseau:vdv:10} and later on in
\citet{ShenTokdarGhosal2013,scricciolo:12} it was proved that location mixture
of Gaussian distributions lead to adaptive (nearly) optimal posterior
concentration rates (for $L_1$ metrics) over collections of H\"older types
functional classes, in the context of density estimation for independently and
identically distributed random variables. Contrarywise, in \citet{Canale2013},
suboptimal posterior concentration rates are derived and the authors obtain
rates that are at best $n^{-\beta/(2\beta+2)}$ up to a $\log n$ term in place of
$n^{-\beta/(2 \beta + 1)}$.  These results are obtained under strong assumptions
on the tail of the true density $f_0$, since it is assumed that
$f_0(x) \lesssim e^{-c|x|^{\tau}}$ when $x$ goes to infinity, for some positive
$c, \tau $.

In \citet{Canale2013}, the authors suggest that location-scale mixtures might
lead to suboptimal posterior concentration rates, for light tail distributions
but might be more robust to tails, since the rate $n^{-\beta/(2\beta+2)}$ is the
minimax estimation rate for density estimation with regularity $\beta$, under
the $L_2$ loss, see \citet{reynaud:rivoirard:tuleau11,Goldenshluger2014}.

The question thus remains open as to how robust to tails mixtures of Gaussian
distributions (either location or location-scale) are.

Interestingly in \citet{bochkina:rousseau:16}, much weaker tail constraints are
necessary to achieve the minimax rate $n^{-\beta /(2\beta +1)}$, for estimating
densities on $\mathbb R^+$ using mixtures of Gamma distributions. The authors
merely require that $F_0$ allows for a moment of order striclty greater than
2. However in \citet{bochkina:rousseau:16} as well as in
\citet{kruijer:rousseau:vdv:10,ShenTokdarGhosal2013,scricciolo:12}, the
smoothness functional classes are non standard and roughly correspond to
requiring that the log-density is locally H\"older, which blurs the
understanding of the robustness of Gaussian mixtures to tails. These smoothness
conditions are required to ensure that the density $f_0$ can be approximated by
a mixture $f_{\sigma, G} $ where $G$ is a probability measure in terms of
Kullback-divergence. Hence to better understand the ability of mixture models to
capture heavy tails we study their use in nonparametric regression models:
\begin{equation}
  \label{eq:model}
  \begin{split}
    Y_i &= f(X_i) + \epsilon_i,\quad \epsilon_i \stackrel{i.i.d}{\sim} N(0,s^2),
    \quad i=1,\dots,n,\\
    X_1,\dots,X_n &\iid Q_0, \quad f \in L^2(Q_0).
  \end{split}
\end{equation}

The parameter is $f$ with prior distribution denoted by $\Pi$. We assume that
$s$ is known, which is just a matter of convenience for proofs. All the results
of the paper can be translated to the case $s$ unknown using the same
methodology as \citet{Salomond2013} or \citet{NauletBarat2015}. Our aim is to
study posterior concentration rates in $L^2(Q_0)$ around the true regression
function $f_0$ defined by sequences $\epsilon_n$ converging to zero with $n$ and
such that
\begin{equation}
  \label{eq:rate}
  \Pi\left( d_n(f,f_0) \leq \epsilon_n \mid
    \mathbf y^n, \mathbf x^n \right) = 1 + o_p(1),
\end{equation}
under the model $f_0$, where $d_n$ is the empirical $\ell_2$ distance of the
covariates, defined as
$d_n(f,f_0)^2 := n^{-1}\sum_{i=1}^n|f(x_i) - f_0(x_i)|^2$. By analogy to the
case of density estimation of \citet{reynaud:rivoirard:tuleau11} and
\citet{Goldenshluger2014} we assume that $f_0 \in L^1$ and belongs to a H\"older
ball with smoothness $\beta$. The tail condition are then on the design
distribution and written as $\int_{\mathbb R} |x|^p dQ_0(x) <+\infty$ and our
aim is to study the posterior concentration rate \eqref{eq:rate} for both
location and location-scale mixtures.

We show in \cref{sec:post-conv-rates}, that in most cases location mixtures have
a better posterior concentration rate than location-scale mixtures and unless
$p $ goes to infinity the posterior concentration rates is not as good as the
usual $n^{-\beta/(2\beta+1)}$. This rate is suboptimal for light tail design
points, since in this case the minimax posterior concentration rate is given by
$n^{-\beta/(2\beta+1)}$. To improve on this rate we propose a new version of
location-scale mixture models, which we call the hybrid location-scale mixture
and we show that this nonparametric mixture model leads to better posterior
concentration rates than the location mixture (and thus than the location-scale
mixture). All these results are up to $\log n$ terms. The results are summarized
in \cref{tab:summary} which displays the value $q$ defined by
$\epsilon_n^2 = n^{-q}$.
\begin{table}[!htb]
  \centering
  \caption{Summary of posterior rates of convergence for different types of
    mixtures. The rates are understood to be in the form $\epsilon_n^2 =
    n^{-q}$, up to powers of $\log n$ factors, where $q$ is given below.}
  \label{tab:summary}
  \resizebox{\columnwidth}{!}{%
  \begin{tabular}{l|c|c|c|c}
    \toprule
    & \multicolumn{2}{c|}{$0 < p < 2$}
    & \multicolumn{2}{c}{$p \geq 2$}\\
    \midrule
    & $p < 2\beta/(\beta+1)$
    & $p \geq 2\beta/(\beta+1)$
    & $p  < 2\beta$
    & $p \geq 2\beta$\\
    \midrule
    Location
    & $\displaystyle\frac{2\beta}{3\beta + 1}$
    & $\displaystyle\frac{2\beta}{3\beta + 1}$
    & $\displaystyle\frac{2\beta}{2\beta + 1 + 2\beta/p}$
    & $\displaystyle\frac{2\beta}{2\beta + 1 + 2\beta/p}$\\
    Location-scale
    & $\displaystyle\frac{2\beta}{3\beta + 2}$
    & $\displaystyle\frac{2\beta}{2\beta + 1 + 2\beta/p}$
    & $\displaystyle\frac{2\beta}{2\beta + 1 + 2\beta/p}$
    & $\displaystyle\frac{\beta}{\beta + 1}$\\
    Hybrid
    & $\displaystyle\frac{2\beta}{3\beta + 1}$
    & $\displaystyle\frac{p}{p + 1}$
    & $\displaystyle\frac{p}{p + 1}$
    & $\displaystyle\frac{2\beta}{2\beta + 1}$\\
    \bottomrule
  \end{tabular}
  }
\end{table}

Although the results are presented in the regression model, we believe that
similar phenomena should take place in the density estimation problem.

The main results with the description of the three types of prior models and the
associated posterior concentration rates are presented in
\cref{sec:post-conv-rates}. Proofs are presented in \cref{sec:bound-post-distr}
and some technical lemmas are proved in the appendix.

\subsection{Notations}
\label{sec:notations}

\par We call $P_f(\cdot \mid X)$ the distribution of the random variable
$Y \mid X$ under the model \eqref{eq:model}, associated with the regression
function $f$. Given $(X_1,\dots,X_{n})$, $P_f^n(\cdot \mid X_1,\dots,X_n)$
stands for the distribution of the random vector $(Y_1,\dots,Y_n)$ of
independent random variables $Y_j \sim P_f(\cdot \mid X_j)$. Also, for any
random variable $X$ with distribution $P$, and any function $g$, $Pg(X)$ denote
the expectation of $g(X)$.

\par For any $\alpha > 0$, we let $\mathrm{SGa}(\alpha)$ denote the symmetric
Gamma distribution with parameter $\alpha$; that is $X \sim
\mathrm{SGa}(\alpha)$ has the distribution of the difference of two independent
Gamma random variables with parameters $(\alpha,1)$.

\par For any finite positive measure $\alpha$ on the measurable space
$(X,\mathcal{X})$, let $\Pi_{\alpha}$ denote the symmetric Gamma process
distribution with parameter $\alpha$ \citep{WolpertClydeTu2011,NauletBarat2015};
that is, an $M \sim \Pi_{\alpha}$ is a random signed measure on
$(X,\mathcal{X})$ such that far any disjoints $B_1,\dots, B_k \in \mathcal{X}$
the random variables $M(B_1),\dots,M(B_k)$ are independent with distributions
$\mathrm{SGa}(\alpha(B_i))$, $i=1,\dots,k$.

\par For any $\beta > 0$, we let $\mathtt C^{\beta}$ denote the H\"older space
of order $\beta$; that is the set of all functions
$f : \mathbb R \rightarrow \mathbb R$ that have bounded derivatives up to order
$m$, the largest integer smaller than $\beta$, and such that the norm
$\|f\|_{\mathtt C^{\beta}} := \sup_{k\leq m} \sup_{x\in \mathbb R}|f^{(k)}(x)| +
\sup_{x\ne y}|f^{(m)}(x) - f^{(m)}(y)|/|x - y|^{\beta - m}$ is finite.

\par For $1 \leq p < \infty$ we let $L^p$ be the space of function for which the
norm $\|f\|_p^p := \int |f(x)|^p\,dx$ is finite; and by $L^{\infty}$ we mean the
space of functions for which $\|f\|_{\infty} := \sup_{x\in \mathbb R}|f(x)|$ is
finite. For $0 \leq p,q\leq \infty$ and functions $f \in L^p$, $g\in L^q$, we
write $f*g$ the convolution of $f$ and $g$, that is $f*g(x) := \int f(x
-y)g(y)\, dy$ for all $x\in \mathbb R$. Moreover, we'll use repeatedly Young's
inequality which state that $\|f*g\|_r \leq \|f\|_p\|g\|_q$, with $1/p + 1/q =
1/r + 1$.

\par If $f\in L^1$, then we define $\widehat{f}$ as the $(L^1)$ Fourier
transform of $f$; that is $\widehat{f}(\xi) := \int f(x)e^{-i\xi x}\, dx$ for
all $\xi \in \mathbb R$. Moreover, if $\widehat{f} \in L^1$, then the inverse
Fourier transform is well-defined and
$f(x) = (2\pi)^{-1}\int \widehat{f}(\xi) e^{ix \xi}\, d\xi$. Also, we denote by
$\mathcal{S}$ the Schwartz space; that is the space of infinitely differentiable
functions $f : \mathbb R \rightarrow \mathbb R$ for which
$|x^{r}f^{(k)}(x)| < +\infty$ for all $r > 0$ and all $k \in \mathbb N$. Then
$\mathcal{S} \subset L^1$, and it is well known that the Fourier transform maps
$\mathcal{S}$ onto itself, thus the Fourier transform is always invertible on
$\mathcal{S}$. We note
$\|f\|_{r,k} = \sup\{|x|^r |f^{(k)}(x) | , x \in \mathbb R \} $ for any
$f \in \mathcal S$.

\par For two real numbers $a,b$, the notation $a \wedge b$ stand for the minimum
of $a$ and $b$ whereas $a \vee b$ stand for the maximum. Similarly, given two
real valued functions $f,g$ the function $f\wedge g$ is the function which at
$x$ assigns the minimum of $f(x)$ and $g(x)$ and $f\vee g$ has obvious
definition. Throughout the paper $C$ denotes a generic constant.

\par Inequalities up to a generic constant are denoted by $\lesssim$ and
$\gtrsim$.

\section{Posterior convergence rates for Symmetric Gamma mixtures}
\label{sec:post-conv-rates}

In this section we present the main results of the paper. We first present the
three types of priors that are studied; i.e. location mixtures, location - scale
mixtures and hybrid location-scale mixtures and for each of these families of
priors we provide the associated posterior concentration rates.

Recall that we consider observations $(Y_i,X_i)_{i=1}^n$ independent and
identically distributed according to model \eqref{eq:model} and we note
$\mathbf y^n = (Y_1, \cdots, Y_n)$ and $\mathbf x^n = (X_1, \cdots, X_n)$. We
denote the prior and the posterior distribution on $f$ by $\Pi(\cdot ) $ and
$\Pi(\cdot \mid \mathbf y^n, \mathbf x^n ) $ respectively.

\subsection{Family of priors}
\label{sec:family-priors}

\subsubsection{Location mixtures of Gaussians}
\label{sec:locat-mixt-gauss}

\par A symmetric Gamma process location mixture of Gaussians prior $\Pi$ is the
distribution of the random function $f(x) := \int \varphi((x - \mu)/\sigma)\,
dM(\mu)$ where $\sigma \sim G_{\sigma}$ and $M \sim \Pi_{\alpha}$, with $\alpha$
a finite positive measure on $\mathbb R$, $G_{\sigma}$ a probability measure
on $(0,\infty)$ and $\varphi(x) := e^{-x^2/2}$ for all $x\in \mathbb R$.

\par We restrict our discussion to priors for which the following conditions are
verified. We assume that there are positive constants $a_1,a_2,a_{3}$ and
$b_1,b_2,b_3,b_4$ such that $G_{\sigma}$ satisfies for $x \geq 1$
\begin{gather}
  \label{eq:15}
  G_{\sigma}\left(\sigma > x\right) \lesssim \exp( - a_1 x^{b_1})\\
  \label{eq:16}
  G_{\sigma}\left(\sigma \leq 1/x\right) \lesssim \exp(- a_2 x^{b_2})\\
  \label{eq:17}
  G_{\sigma}\left(x^{-1} \leq \sigma \leq x^{-1}(1 + t)\right) \gtrsim x^{b_3}
  t^{b_4}\exp(-a_3x),\quad \forall t\in (0,1).
\end{gather}
We let $\alpha := \overline{\alpha} G_{\mu}$ for a positive constant
$\overline{\alpha} > 0$ and $G_{\mu}$ a probability distribution on $\mathbb
R$. We assume that there are positive constants $b_5,b_6$ such that $G_{\mu}$
satisfies for all $x \in \mathbb R$
\begin{equation}
  \label{eq:18}
  G_{\mu}\left(|\mu - x| \leq t\right) \gtrsim t^{b_5}(1 +
  |x|)^{-b_6}, \quad \forall t  \in (0,1).
\end{equation}
The heavy tail condition on $G_{\mu}$ is required to not deteriorate the rate of
convergence when $Q_0$ is heavy tailed.

\par Notice that \cref{eq:15} forbids the use of the classical inverse-Gamma
distribution as prior distribution on $\sigma$ because of its heavy tail. In
fact, it is always possible to weaken \cref{eq:15} to allow for Inverse-Gamma
distribution (see \cite{Canale2013,NauletBarat2015}) but it complicates the
proofs with no contribution to the subject of the paper. We found that among the
usual distributions the inverse-Gaussian is more suitable for our purpose since
it fulfills all the \cref{eq:15,eq:16,eq:17}, as shown in \cref{pro:11}. We
recall that the inverse-Gaussian distribution on $(0,\infty)$ with parameters
$a>0$, $b > 0$ has density with respect to Lebesgue measure
\begin{equation*}
  f(x;a,b)
  := \left(\frac{b}{2\pi x^3}\right)^{1/2} \exp\left(-\frac{b(x -a)^2}{2a^2 x}
  \right),\quad \forall x > 0,
\end{equation*}
and $f(x;a,b) = 0$ elsewhere.

\begin{proposition}
  \label{pro:11}
  The inverse-Gaussian distribution with parameters $b,a > 0$
  satisfies \cref{eq:15,eq:16,eq:17} with $a_1 = b / (2a^2)$, $b_1 = 1$,
  $a_2 = b / 4$, $b_2 = 1$, $b_3 = 1$, $b_4 =1$ and
  $a_3 = b / 2$.
\end{proposition}
\begin{proof}
  It suffices to write, for any $x\geq 1$
  \begin{align*}
    G_{\sigma}\left(\sigma > x\right)
    &\leq \left(\frac{b}{2\pi x^3}\right)^{1/2} \int_x^{\infty}
      \exp\left(-\frac{b(t - a)^2}{2a^2t}\right)\,dt\\
    &\leq \left(\frac{b}{2\pi}\right)^{1/2} \exp\left( \frac{b}{a}
      -  \frac{b}{2}\right) \int_x^{\infty}\exp\left(-\frac{b
      t}{2a^2}\right)\, dt.
  \end{align*}
  Also, for any $x\geq 1$
  \begin{align*}
    G_{\sigma}\left(\sigma \leq 1/x\right)
    &\leq \left(\frac{b}{2\pi}\right)^{1/2} \int_0^{1/x}
      t^{-3/2}\exp\left(-\frac{b(t - a)^2}{2a^2t}\right)\,dt\\
    &\leq \left(\frac{b}{2\pi}\right)^{1/2} e^{b/a} \int_0^{1/x}
      t^{-3/2}e^{-b/(2t)}\, dt\\
    &\leq 216(b \sqrt e)^{-3}\left(\frac{b}{2\pi}\right)^{1/2}
      e^{b/a} \int_0^{1/x} e^{-b/(4t)}\, dt.
  \end{align*}
  Finally, for any $x \geq 1$ and $0 < t < 1$,
  \begin{align*}
    G_{\sigma}\left(x^{-1} \leq \sigma \leq x^{-1}( 1 +t) \right)
    &\geq \left(\frac{b}{2\pi}\right)^{1/2} e^{b/a-b/a^2}
      \int_{x^{-1}}^{x^{-1}(1 + t)}e^{-b/(2t)}\,dt. \qedhere
  \end{align*}
\end{proof}

\subsubsection{Location-scale mixtures of Gaussians}
\label{sec:locat-scale-mixt}

\par A symmetric Gamma process location-scale mixture of Gaussians prior $\Pi$
is the distribution of the random function
$f(x) := \int \varphi((x - \mu)/\sigma)\, dM(\sigma, \mu)$ where
$M \sim \Pi_{\alpha}$, with $\alpha$ a finite positive measure on
$(0,\infty) \times \mathbb R$ and $\varphi(x) := e^{-x^2/2}$ for all
$x\in \mathbb R$. We focus the attention of the reader on the fact that
althought we use the same notations (\textit{i.e.} $\Pi$, $\alpha$) as the
previous section, these are different distributions and in the sequel we pay
attention as making the context clear enough to avoid confusions.

\par We restrict our discussion to priors for which
$\alpha := \overline{\alpha}G_{\sigma} \times G_{\mu}$, with
$\overline{\alpha} > 0$ and $G_{\sigma}$, $G_{\mu}$ satisfying the same
assumptions as in \cref{sec:locat-mixt-gauss}.

\subsubsection{Hybrid location-scale mixtures of Gaussians}
\label{sec:hybr-locat-scale}

\par The proof of the results given in the two preceeding sections suggests that
neither location or location-scale mixtures can achieve the optimal rates,
whatever the nature of the tails of $Q_0$. We show that we can get better upper
bounds by introducing hybrid mixtures.

\par By a hybrid location-scale mixtures of Gaussians, we mean the distribution
$\Pi$ of the random function
$f(x) := \int \varphi((x - \mu)/\sigma) \, dM(\sigma,\mu)$, where
$M \sim \Pi_{\alpha}$, with
$\alpha = \overline{\alpha}P_{\sigma}\times G_{\mu}$, $\overline{\alpha} > 0$,
$P_{\sigma} \sim \Pi_{\sigma}$ and $G_{\mu}$ a probability measure satisfying
\cref{eq:18}. Here $\Pi_{\sigma}$ is a prior distribution on the space of
probability measures (endowed with Borel $\sigma$-algebra). We now formulate
conditions on $\Pi_{\sigma}$ that are the random analoguous to
\cref{eq:15,eq:16}. For the same constants $a_1,a_2,b_1,b_2$ as in
\cref{sec:locat-mixt-gauss}, we consider the existence of positive constants
$a_4,a_5$ such that $\Pi_{\sigma}$ satisfies for $x>0$ large enough
\begin{gather}
  \label{eq:4}
  \Pi_{\sigma}\left( P_{\sigma} : P_{\sigma}(\sigma > x)
    \geq \exp(-a_1x^{b_1} /2) \right) \lesssim \exp(-a_4 x^{b_1}),\\
  \label{eq:5}
  \Pi_{\sigma}\left( P_{\sigma} : P_{\sigma}(\sigma < 1/x)
    \geq \exp(-a_2x^{b_2} /2) \right) \lesssim \exp(-a_5 x^{b_2}).
\end{gather}
As a replacement of \cref{eq:17}, we assume that for all $r \geq 1$ there are
constants $a_6,b_7$ such that for any positive integer $J$ large enough
\begin{equation}
  \label{eq:6}
  \Pi_{\sigma}\left(\cap_{j=0}^{J}\Set{P_{\sigma} \given
      P_{\sigma}[2^{-j}, 2^{-j}(1+2^{-Jr})] \geq 2^{-J}}\right)
  \gtrsim \exp(-a_6J^{b_7}2^J).
\end{equation}

\Cref{eq:4,eq:5,eq:6} are rather restrictive and it is not clear \textit{a
  priori} whether or not such distribution exists. For example, if $P_{\sigma}$
is chosen to be almost-surely an Inverse-Gaussian distribution with parameters
$b,\mu$ then \cref{eq:6} is not satisfied. However, we now show that under
conditions on the base measure, $\Pi_{\sigma}$ can be chosen as a Dirichlet
Process, hereafter referred to as DP.

We recall that if $\Pi_{\sigma}$ is a Dirichlet Process distribution with base
measure $\alpha_{\sigma}G(\cdot)$ on $(0,\infty)$ \citep{Ferguson1973}, then
$P_{\sigma} \sim \Pi_{\sigma}$ is a random probability measure on $(0,\infty)$
such that for any Borel measurable partition $A_1,\dots,A_k$ of $(0,\infty)$,
the joint distribution $P_{\sigma}(A_1),\dots,P_{\sigma}(A_k)$ is the
$k$-variate Dirichlet distribution with parameters
$\alpha_{\sigma}G(A_1),\dots,\alpha_{\sigma}G(A_k)$.

\begin{proposition}
  \label{pro:14}
  Let $\alpha_{\sigma}> 0$, $G_{\sigma}$ a probability measure on $(0,\infty)$
  satisfying the same assumptions as in \cref{eq:15,eq:16,eq:17}, and
  $\Pi_{\sigma}$ be a Dirichlet Process with base measure
  $\alpha_{\sigma}G_{\sigma}(\cdot)$. Then $\Pi_{\sigma}$ satisfies
  \cref{eq:4,eq:5,eq:6} with constants $a_4 = a_1$, $a_5 = a_2$, a
  constant $a_6 > 0$ eventually depending on $r$, and $b_7 = 0$.
\end{proposition}
\begin{proof}
  \par We first prove \cref{eq:4}. It follows from the definition of the DP that
  $P_{\sigma}(x,\infty)$ has Beta distribution with parameters
  $\alpha_{\sigma}G_{\sigma}(x,\infty)$ and
  $\alpha_{\sigma}(1-G_{\sigma}(x,\infty))$, then by Markov's inequality
  \begin{align*}
    \Pi_{\sigma}\Big( P_{\sigma} : P_{\sigma}(x,\infty) \geq t\Big)
    \leq \frac{G_{\sigma}(x,\infty)}{t}.
  \end{align*}
  Likewise, if $t= \exp(-a_1x^{b_1}/2)$ and $G_{\sigma}$ satisfies
  \cref{eq:15,eq:16,eq:17}, the conclusion follows. The same steps with
  $G_{\sigma}(0,1/x)$ give the proof of \cref{eq:5}. It remains to prove
  \cref{eq:6}. Let $r\geq 1$ and define
  $V_{j,r} := \Set{\sigma \given 2^{-j} \leq \sigma \leq 2^{-j}(1 + 2^{-Jr}) }$
  for any integer $0 \leq j \leq J$. For all $r\geq 1$ the $V_{j,r}$'s are
  disjoint. Set $V_r^c := \cup_{j=0}^JV_{j,r}^c$. If
  $\alpha_{\sigma}G_{\sigma}(V_r^c) \leq 1$ let $V_{J+1,r} = V_r^c$ and $M=1$ ;
  otherwise split $V_r^c$ into $M>1$ disjoint subsets
  $V_{1,r}^c,\dots V_{M,r}^c$ such that
  $\exp(-2^J) \leq \alpha_{\sigma}G_{\sigma}(V_{k,r}^c) \leq 1$ for all
  $k=1,\dots,M$ and set $V_{J+1,r} = V_{1,r}^c$, $V_{J+2,r}=V_{2,r}^c$,
  $\dots, V_{J+M,r} = V_{M,r}^c$ (since $G_{\sigma}(0,\infty) = 1$ this can be
  done with a number $M$ independent of $J$). For $J$ large enough (so that
  $(J+M)2^{-J+1} < 1$), acting as in \citet[lemma~6.1]{Ghosal2000}, it follows
  \begin{equation*}
     \Pi_{\sigma}\Big(P_{\sigma} : P_{\sigma}[2^{-j},2^{-j}(1+2^{-Jr}] \geq
    2^{-J} \quad \forall\ 0 \leq j \leq J\Big)
    \geq \frac{\Gamma(\alpha_{\sigma}) 2^{-J(J+M)}}
    {\prod_{j=0}^{J + M}\Gamma(\alpha_{\sigma}G_{\sigma}(V_{j,r}))},
  \end{equation*}
  Also, $\alpha_{\sigma}G_{\sigma}(V_{j,r}) \leq 1$ implies
  $\Gamma(\alpha_{\sigma}G_{\sigma}(V_{j,r})) \leq
  1/(\alpha_{\sigma}G_{\sigma}(V_{j,r}))$, hence
  \begin{multline*}
    \Pi_{\sigma}\Big(  P_{\sigma}\ :\ P_{\sigma}[2^{-j},2^{-j}(1+2^{-Jr}] \geq
    2^{-J} \quad \forall\ 0 \leq j \leq J\Big)\\
    \geq \Gamma(\alpha_{\sigma}) \alpha_{\sigma}^{J+M+1} 2^{-J(J+M)}
    \prod_{j=0}^{J+M}G_{\sigma}(V_{j,r}).
  \end{multline*}
  Since $M$ does not depend on $J$, one can find a constant $C >0$ such that
    \begin{multline*}
      \Pi_{\sigma}\Big(P_{\sigma} : P_{\sigma}[2^{-j},2^{-j}(1+2^{-Jr}] \geq
      2^{-J} \quad \forall\ 0 \leq j \leq J\Big)\\ \geq
      \Gamma(\alpha_{\sigma})\exp\left\{ -C J^2 + \sum_{j=0}^J\log
        G_{\sigma}(V_{j,r}) + \sum_{j=J+1}^{J+M}\log G_{\sigma}(V_{j,r})
      \right\}.
  \end{multline*}
  By construction, the second sum in the rhs of the last equation is lower
  bounded by $-M 2^J$, whereas if $G_{\sigma}$ satisfies
  \cref{eq:15,eq:16,eq:17}, the first sum is lower bounded by $-C'2^J$ for a
  constant $C'>0$ eventually depending on $r$. Then the proposition is
  proved.
\end{proof}

\subsection{Posterior concentration rates under the mixture priors}
\label{sec:results}

We let $\Pi(\cdot \mid \mathbf y^n , \mathbf x^n)$ denote the posterior
distribution of $f\sim \Pi$ based on $n$ observations
$(X_1,Y_1),\dots,(X_n,Y_n)$ modelled as in \cref{sec:introduction}. Let
$(\epsilon_n)_{n\geq 1}$ be a sequence of positive numbers with
$\lim_n\epsilon_n = 0$, and $d_n$ denote the empirical $L^2$ distance, that is
$nd_n(f,g)^2 = \sum_{i=1}^n|f(X_i) - g(X_i)|^2$.

The following theorem is proved in Section \cref{sec:bound-post-distr} .

\begin{theorem}
  \label{thm:bigthm}
  Consider the model \eqref{eq:model}, and assume that
  $f_0 \in L^1 \cap \mathtt C^{\beta}$ and $Q_0|X|^p < +\infty$. Then there
  exist a constant $C > 0$ and $t > 0$ depending only on $f_0$ and $Q_0$ such
  that
  \begin{itemize}
    \item If the prior $\Pi$ is the symmetric Gamma location mixture of
    Gaussians as defined in \cref{sec:locat-mixt-gauss}
    \begin{equation*}
      \Pi\left( d_n( f , f_0)^2 > Cn^{-2\beta/(3\beta + 1)}(\log n)^t \mid
        \mathbf y^n , \mathbf x^n\right) = o_p(1)
    \end{equation*}
    when $0 < p \leq 2$, and
    \begin{equation*}
      \Pi\left( d_n( f , f_0)^2 > Cn^{-2\beta/(2\beta + 1 + 2\beta/p)}(\log n)^t
        \mid \mathbf y^n , \mathbf x^n\right) = o_p(1)
    \end{equation*}
    when $p > 2$.
    \item If the prior $\Pi$ is the symmetric Gamma location-scale mixture of
    Gaussians defined in \cref{sec:locat-scale-mixt}
    \begin{equation*}
      \Pi\left( d_n( f , f_0)^2 > C[n^{-2\beta/(3\beta +2)}\wedge
        n^{-2\beta/(2\beta + 1 + 2\beta/p)}](\log n)^t\mid \mathbf y^n , \mathbf
        x^n\right) =  o_p(1)
    \end{equation*}
    when $0 < p \leq 2$, and
    \begin{equation*}
      \Pi\left( d_n( f , f_0)^2 > Cn^{-\beta/(\beta + 1)}(\log n)^t \mid \mathbf
        y^n, \mathbf x^n\right) = o_p(1),
    \end{equation*}
    when $p > 2\beta$.
    \item If the prior $\Pi$ is the hybrid symmetric Gamma location-scale
    mixture of Gaussians defined in \cref{sec:hybr-locat-scale}
    \begin{equation*}
      \Pi\left( d_n( f , f_0)^2 > C[n^{-2\beta/(3\beta + 1)}\wedge
        n^{-p/(p+1)}](\log n)^{t} \mid \mathbf y^n , \mathbf x^n\right) =
      o_p(1),
    \end{equation*}
    when $p \leq 2\beta$ or
    \begin{equation*}
      \Pi\left( d_n( f , f_0)^2 > Cn^{-2\beta/(2\beta + 1)}(\log n)^{t}\mid
        \mathbf y^n , \mathbf x^n\right) = o_p(1),
    \end{equation*}
    when $p > 2\beta$ .
  \end{itemize}
\end{theorem}
The upper bounds on the rates in the previous paragraph are no longer valid when
$p=0$. Indeed the constant $C > 0$ depends on $p$ and might not be definite if
$p=0$ ; the reason is to be found in the fact that $C$ heavily depends on the
ability of the prior to draw mixture component in regions of observed data,
which remains concentrated near the origin when $p > 0$. In
\cref{sec:relax-tail-assumpt}, we overcome this issue by making the prior
covariate dependent ; this allows to derive rates under the assumption $p=0$ (no
tail assumption).

\subsection{Relaxing the tail assumption : covariate dependent prior for
  location mixtures}
\label{sec:relax-tail-assumpt}

Although the rates derived in \cref{sec:bound-post-distr} do not depend on
$p > 0$ when $p$ is small, the assumption $Q_0|X|^p < +\infty$ is crucial in
proving the Kullback-Leibler condition. Indeed, this condition ensures that the
covariates belong to a set $\mathcal X_n$ which is not too large, which allows
us to bound from below the prior mass of Kullback-Leibler neighbourhoods of the
true distribution. Surprisingly, it seems very difficult to get rid of this
assumption under a fully Bayesian framework without fancy assumptions, while
making the prior covariates dependent allows to drop all tail conditions on
$Q_0$. Doing so, we can adapt to the tail behaviour of $Q_0$, as shown in the
following theorem, which is an adaptation of the general theorems of
\citet{GhosalVanDerVaartothers2007}. For convenience, in the sequel we drop out
the superscript $n$ and we write $\mathbf x$, $\mathbf y$ for $\mathbf x^n$,
$\mathbf y^n$, respectively. For $\epsilon > 0$ and anu subset $A$ of a metric
space equipped with metric $d$, we let $N(\epsilon,A,d)$ denote the
$\epsilon$-covering number of $A$, \textit{i.e.} $N(\epsilon,A,d)$ is the
smallest number of balls of radius $\epsilon$ needed to cover $A$.

\begin{theorem}
  \label{thm:1}
  Let $\Pi_{\mathbf x}$ be a prior distribution that depends on the covariate
  vector $\mathbf x$, $0 < c_2 < 1/4$ and $\epsilon_n \rightarrow 0$ with
  $n\epsilon_n^2 \rightarrow \infty$. Suppose that
  $\mathcal{F}_n \subseteq \mathcal{F}$ is such that
  $Q_0^n\Pi_{\mathbf x}(\mathcal{F}_n^c) \lesssim \exp(-\frac
  12(1+2c_2)n\epsilon_n^2)$ and
  $\log N(\epsilon_n/18, \mathcal{F}_n, d_n) \leq n\epsilon_n^2/4$ for $n$ large
  enough. If for any $\mathbf x \in \mathbb R^n$ it holds
  $\Pi_{\mathbf x}(f\ : \ d_n(f,f_0) \leq s \epsilon_n)\gtrsim
  \exp(-c_2n\epsilon_n^2)$, then for all $M > 0$ we have
  $\Pi_{\mathbf x}(f\ :\ d_n(f,f_0) > M \epsilon_n \mid \mathbf y, \mathbf x) =
  o_p(1)$.
\end{theorem}

We apply \cref{thm:1} to symmetric Gamma process location mixtures of Gaussians
in the following way. Let $\mathbb Q_{\mathbf x}^n$ denote the empirical measure
of the covariate vector $\mathbf x$. Given a a probability density function $g$,
we let $G_{\mathbf x}$ the probability measure which density is
$z\mapsto \int g(z - x_i)\, d\mathbb Q_{\mathbf x}^n(x)$.

\begin{corollary}
  \label{cor:locmix:covar}
  Then we let $\Pi_{\mathbf x}$ be the distribution of the random function
  $f(x) := \int \varphi((x-\mu)/\sigma)\,dM(\mu)$, where
  $\sigma \sim G_{\sigma}$ and $M \sim \Pi_{\alpha}$ with
  $\alpha = \overline{\alpha} G_{\mathbf x}$ for some $\overline{\alpha} >
  0$. Assume that $G_{\sigma}$ satisfies \cref{eq:15,eq:16,eq:17} and that there
  exists a constant $b_8 > 0$ such that
  $\sup_{\mathbf x \in \mathbb R^n}G_{\mathbf x}(\mu \, :\, |\mu - s| \leq t)
  \lesssim t^{b_8}$ for all $0 < t,s \leq 1$. Then
  $\Pi_{\mathbf x}(f\ :\ d_n(f,f_0) > M \epsilon_n \mid \mathbf y, \mathbf x) =
  o_p(1)$ with
  $\epsilon_n^2 \lesssim n^{-2\beta/(3\beta+1)}(\log n)^{2-2\beta/(3\beta+1)}$.
\end{corollary}

To prove \cref{cor:locmix:covar}, note that neither the proof of \cref{lem:3} or
\cref{lem:5} involve the base measure $\alpha$ (indeed, it only involves
$\overline{\alpha}$); thus we can use the sieve $\mathcal{F}_n$ constructed in
\cref{sec:construction-tests-1}. To apply \cref{thm:1} it is then sufficient to
prove that for all $x\in \mathbb R^n$
\begin{equation}
  \label{eq:22}
  \Pi_x(f\ :\ d_n(f,f_0) \leq s \epsilon_n) \gtrsim \exp(-c_2n
  \epsilon_n^2).
\end{equation}
This is done in \cref{lem:8}.

\begin{lemma}
  \label{lem:8}
  Assume that there is a constant $b_8 > 0$ such that
  $\sup_{\mathbf x \in \mathbb R^n}G_{\mathbf x}(\mu \, :\, |\mu - s| \leq t)
  \lesssim t^{b_8}$ for all $0 < t,s \leq 1$. Also assume that $G_{\sigma}$
  satifies \cref{eq:15,eq:16,eq:17}. Then \cref{eq:22} holds for the symmetric
  Gamma location mixture of Gaussians with base measure
  $\overline{\alpha} G_{\mathbf x}$ if
  $\epsilon_n^2 \leq C n^{-2\beta/(3\beta+1)}(\log n)^{2-2\beta/(3\beta+1)}$ for
  an appropriate constant $C > 0$.
\end{lemma}

The proof of \cref{lem:8} is given in \cref{sec:prlem8}.

\section{Proofs}
\label{sec:bound-post-distr}

To prove \cref{thm:bigthm} we follow the lines of
\citet{Ghosal2000,Ghosal2001,Ghosal2007}. Namely we need to verify the following
three conditions
\begin{itemize}
  \item Kullback-Leibler condition : For a constant $0 < c_2 < 1/4$,
  \begin{equation}
    \label{eq:KLcond}
    \Pi(\mathtt{KL}(f_0,\epsilon_n)) \geq e^{-c_2 n\epsilon_n^2},
  \end{equation}
  where
  \begin{equation*}
    \mathtt{KL}(f_0,\epsilon_n)
    := \Set*{f \given \frac{1}{2s^2} \int |f_0(x) - f(x)|^2\, dQ_0 (x) \leq
      \epsilon_n^2}.
  \end{equation*}
  \item Sieve condition : There exists $\mathcal F_n \subset \mathcal F$ such
  that
  \begin{equation}
    \label{eq:sievecond}
    \Pi(\mathcal{F}_n^c ) \leq e^{-\frac 12(1+ 2c_2)n\epsilon_n^2 }
  \end{equation}
  \item Tests : Let $N(\epsilon_n / 18, \mathcal{F}_n, d_n)$ be the logarithm of
  the covering number of $\mathcal F_n$ with radius $\epsilon_n / 18$ in the
  $d_n(\cdot,\cdot)$ metric.
  \begin{equation}
    \label{eq:entropycond}
    N(\epsilon_n / 18, \mathcal{F}_n, d_n) \leq \frac{n\epsilon_n^2}{4}.
  \end{equation}
\end{itemize}

The Kullback-Leibler condition is proved by defining an approximation of $f$ by
a discrete mixture under weak tail conditions. Although the general idea is
close to \citet{kruijer:rousseau:vdv:10} or \citet{scricciolo:12}, the
construction remains quite different to be able to handle various tail
behaviours. This is detailed in the following section.

\subsection{Approximation theory}
\label{sec:approximation-theory}

To describe the approximation of $f_0$ by a finite mixture, we first define a
few notations.

Let $\widehat{\chi}$ be a $\mathtt C^{\infty}$ function that equals $1$ on
$[-1,1]$ and $0$ outside $[-2,2]^c$ (think for instance as the convolution of
$\Ind_{[-1,1]}$ with $x\mapsto \exp(-1/(1-x^{2}))\Ind_{[-1,1]}(x)$). For any
$\sigma > 0$ we use the shortened notation
$\widehat{\chi}_{\sigma}(\xi) := \widehat{\chi}(2\sigma\xi)$. Define $\eta$ as
the function which $L^1$ Fourier transform satisfies
$\widehat{\eta}(\xi) = \widehat{\chi}(\xi) / \widehat{\varphi}(\xi)$ for all
$\xi \in [-2,2]$ and $\widehat{\eta}(\xi) = 0$ elsewhere. For two positive real
numbers $h$ and $\sigma$, we define the kernel
$K_{h,\sigma} : \mathbb R \times \mathbb R \rightarrow \mathbb R$ such that
\begin{align*}
  K_{h,\sigma}(x,y) := \frac{h}{\sigma}\sum_{k\in\mathbb Z}\varphi\left(\frac{x
  - h\sigma k}{\sigma} \right) \eta \left(\frac{y - h\sigma k}{\sigma}\right),
  \quad \forall (x,y) \in \mathbb R \times \mathbb R.
\end{align*}
For a measurable function $f$ we introduce the operator associated with the
kernel : $K_{h,\sigma}f(x) = \int K_{h,\sigma}(x,y)f(y)\, dy$. The function
$K_{h,\sigma}f$ will play the role of an \textit{approximation} for the function
$f$, and we will evaluate how this approximation becomes close to $f$ given $h$
and $\sigma$ sufficiently close to zero.

More precisely, we will prove that, when choosing $h$ appropriately, $f$ can be
approximated by $K_{h,\sigma} \chi_\sigma\times f_0 $ to the order
$\sigma^\beta$. Moreover $K_{h,\sigma} \chi_\sigma\times f_0$ can be written as
$\sum_{k \in \mathbb Z} u_k \varphi( (x - \mu_k)/\sigma) )$. In a second step we
approximate $K_{h,\sigma} \chi_\sigma\times f_0 $ by a truncated version of it,
retaining only the $k$'s such that $|u_k|$ is large enough and $|\mu_k|$ not too
large.  In the case of location - scale and hybrid location - scale mixtures we
consider a modification of this approximation to control better the number of
components for which $\sigma $ needs to be small. We believe that these
constructions have interest in themselves. In partcular they shed light on the
relations between Gaussian mixtures and wavelet approximations.

These approximation properties are presented in the following two Lemmas which
are proved in \cref{sec:eta}:
\begin{lemma}
  \label{lem:4}
  There is $C > 0$ depending only on $\beta$ such that for any
  $f_0 \in L^1\cap \mathtt C^{\beta}$ and any $\sigma > 0$ we have
  $|\chi_{\sigma}*f_0(x) - f_0(x)| \leq C\|f\|_{\mathtt C^{\beta}}
  \sigma^{\beta}$ for all $x\in \mathbb R$.
\end{lemma}

\begin{lemma}
  \label{lem:1}
  Let $f_{\sigma} := \chi_{\sigma}*f_0$ and $h \leq 1$. Then there is a
  universal constant $C > 0$ such that
  $|K_{h,\sigma}f_{\sigma}(x) - f_{\sigma}(x)| \leq C \|f_0\|_1 \sigma^{-1}
  e^{-4\pi^2/h^2}$ for all $x\in \mathbb R$.
\end{lemma}

We now present the approximation schemes in the context of location mixtures.

\subsection{Construction of the approximation under location mixtures}
\label{sec:location-mixtures-1}

Let $0 < \sigma \leq 1$ and
$h_{\sigma} \sqrt{\log \sigma^{-1}}:= 2\pi\sqrt{\beta +1}$. Then combining the
results of \cref{lem:4} and \cref{lem:1} we can conclude that
$|K_{h_{\sigma},\sigma}(\chi_{\sigma}*f_0)(x) - f_0(x)| \lesssim
\sigma^{\beta}$. Now we define the coefficients $u_k$, $k\in \mathbb Z$ so that
\begin{align*}
  K_{h_{\sigma},\sigma}(\chi_{\sigma} *f_0)(x) := \sum_{k\in \mathbb Z} u_k\,
  \varphi\left( \frac{x - \mu_k}{\sigma} \right),\quad \forall k \in \mathbb Z,
\end{align*}
where $\mu_k := h_{\sigma}\sigma k$ for all $k\in \mathbb Z$. Let define
\begin{equation*}
  \Lambda := \Set*{k \in \mathbb Z \given
    |u_k| > \sigma^{\beta},\quad |\mu_k| \leq \sigma^{-2\beta/p} +
    \sigma \sqrt{2(\beta + 1)\log\sigma^{-1}}},
\end{equation*}
$U_{\sigma} := \Set{\sigma' \given \sigma \leq \sigma' \leq \sigma(1 +
  \sigma^{\beta})}$, and for all $k\in \Lambda$ we define
$V_k := \Set{\mu \given |\mu - \mu_k| \leq \sigma^{\beta + 1}}$ and
$V = \cup_{k \in \Lambda} V_k$. We also denote
\begin{equation*}
  \mathcal{M}_{\sigma} := \Set*{M \, \mbox{ signed measure on } \mathbb R \given
    \begin{array}{l}
      |M(V_{k}) - u_{k}| \leq \sigma^{\beta},\\
      \forall k \in \Lambda: \,
      |M|(V^c) \leq \sigma^{\beta}
    \end{array}
  },
\end{equation*}
and for any $M\in \mathcal{M}_{\sigma}$, we write
$f_{M,\sigma}(x) := \int \varphi((x - \mu)/\sigma)\, dM(\mu)$.

\begin{proposition}
  \label{pro:13}
  For $\sigma > 0$ small enough, it holds
  $|\Lambda| \lesssim \sigma^{-(\beta+1)} \wedge
  h_{\sigma}^{-1}\sigma^{-(2\beta/p+1)}$.
\end{proposition}
\begin{proof}
  Because there is a separation of $h_{\sigma}\sigma$ between two consecutive
  $\mu_k$, it is clear that
  $|\Lambda| \leq 2h_{\sigma}^{-1}\sigma^{-(2\beta/p + 1)}$. Moreover, from
  \cref{pro:2} we have the following estimate.
  \begin{equation*}
    \|f_0\|_1\sigma^{-1} \gtrsim \sum_{k\in \mathbb Z}|u_k| \geq \sum_{k\in
      \Lambda}|u_k| \geq \sigma^{\beta}|\Lambda|. \qedhere
  \end{equation*}
\end{proof}

\begin{proposition}
  \label{pro:9}
  For all $x \in \mathbb R$, all $\sigma > 0$ small enough and all
  $M \in \mathcal{M}_{\sigma}$ it holds
  $|f_{M,\sigma}(x) - f_0(x)| \lesssim h_{\sigma}^{-1}$.
\end{proposition}
\begin{proof}
  For any $M \in \mathcal{M}_{\sigma}$, we have that
  $|f_{M,\sigma}(x) - f_0(x)| \leq |f_{M,\sigma}(x)| + \|f_0\|_{\infty}$. But,
  with
  $\mathcal{I} \equiv \mathcal{I}(x) := \Set{k \in \mathbb Z \given |x - \mu_k|
    \leq 2\sigma}$,
  \begin{multline}
    \label{eq:12}
    f_{M,\sigma}(x) = \sum_{k\in \Lambda \cap \mathcal{I}}\int_{V_k}
    \varphi\left(
      \frac{x - \mu}{\sigma} \right)\,dM(\mu)\\
    + \sum_{k\in \Lambda \cap \mathcal{I}^c}\int_{V_k} \varphi\left( \frac{x -
        \mu}{\sigma} \right)\,dM(\mu) + \int_{V^c}\varphi\left(\frac{x -
        \mu}{\sigma} \right)\,dM(\mu).
  \end{multline}
  Clearly the last term of this last expression is bounded above by
  $\|\varphi\|_{\infty} \sigma^{\beta}$. For the second term, we have for any
  $\mu \in V_k$ with $k \in \mathcal{I}^c$ that
  $|x - \mu| \geq |x - \mu_k| - |\mu - \mu_k| \geq |x - \mu_k|/2$. Then the
  second term of the rhs of \cref{eq:12} is bounded above by
  \begin{gather*}
    \sup_{k\in \Lambda \cap \mathcal{I}^c}|M|(V_k) \sum_{k\in \mathbb
      Z}\varphi\left( \frac{x - h_{\sigma}\sigma k}{\sigma} \right).
  \end{gather*}
  Proceeding as in the proof of \cref{pro:4}, we deduce that the series in the
  last expression is bounded above by a constant times $1/h_{\sigma}$, whereas
  \cref{pro:2} and Young's inequality yields
  $|M|(V_k) \leq |M(V_k) - u_k| + |u_k| \lesssim \sigma^{\beta} +
  \|\chi_{\sigma}*f_0\|_{\infty} \leq \sigma^{\beta} +
  \|\chi\|_1\|f_0\|_{\infty}$. Therefore the second term of the rhs in
  \cref{eq:12} is bounded by a constant multiple of $h_{\sigma}^{-1}$. Regarding
  the first term in \cref{eq:12}, it is  bounded by
  $\|\varphi\|_{\infty}|\mathcal{I}| \sup_{k\in \Lambda}|M|(V_k)$, which is in
  turn bounded by $h_{\sigma}^{-1}$ times a constant.
\end{proof}

\begin{proposition}
  \label{pro:10}
  For all $\sigma > 0$ small enough, all $x\in \mathbb R$ with
  $|x| \leq \sigma^{2\beta/p}$ and all $M \in \mathcal{M}_{\sigma}$ it holds
  $|f_{M,\sigma}(x) - f_0(x)| \lesssim h_{\sigma}^{-2}\sigma^{\beta}$.
\end{proposition}
\begin{proof}
  We define
  $A_{\sigma}(\beta) := \sqrt{2\log |\Lambda| + 2(\beta + 1)\log
    \sigma^{-1}}$. Then for any $M\in \mathcal{M}_{\sigma}$, letting
  $\mathcal{J} \equiv \mathcal{J}(x) := \Set{k \in \mathbb Z \given |x - \mu_k|
    \leq 2\sigma A_{\sigma}(\beta)}$, we may write
  \begin{multline}
    \label{eq:14}
    f_{M,\sigma}(x) - K_{h_{\sigma},\sigma}(\chi_{\sigma}*f_0)(x) = \sum_{k\in
      \Lambda \cap \mathcal{J}}\int_{V_k}\left[ \varphi\left( \frac{x -
          \mu}{\sigma} \right)
      - \varphi\left(\frac{x - \mu_k}{\sigma} \right) \right]\,dM(\mu)\\
    + \sum_{k\in \Lambda \cap \mathcal{J}}\left[M(V_k) - u_k \right]
    \varphi\left( \frac{x - \mu_k}{\sigma} \right) + \sum_{k\in \Lambda \cap
      \mathcal{J}^c}\int_{V_k} \varphi\left( \frac{x - \mu}{\sigma}
    \right)\,dM(\mu)\\
    - \sum_{k\in \Lambda \cap \mathcal{J}^c}u_k\, \varphi\left( \frac{x -
        \mu_k}{\sigma} \right) - \sum_{k\in \Lambda^c }u_k\, \varphi\left(
      \frac{x - \mu_k}{\sigma} \right)
    + \int_{V^c}\varphi\left(\frac{x - \mu}{\sigma} \right)\,dM(\mu)\\
    := r_1(x) + r_2(x) + r_3(x) + r_4(x) + r_5(x) + r_6(x).
  \end{multline}
  With the same argument as in \cref{pro:13}, we deduce that
  $|\mathcal{J}| \leq 2h_{\sigma}^{-1}A_{\sigma}(\beta)$. The same proposition
  implies $A_{\sigma}(\beta) \lesssim \sqrt{\log\sigma^{-1}}$. Recalling that
  $|M|(V_k) \lesssim 1 + \|\chi\|_1\|f_0\|_{\infty}$ for all $k\in \Lambda$ and
  all $M\in \mathcal{M}_{\sigma}$, it follows from \cref{pro:1} that
  $|r_1(x)| \lesssim A_{\sigma}(\beta) h_{\sigma}^{-1} \sigma^{\beta} $. From
  the definition of $\mathcal{M}_{\sigma}$, it comes
  $|r_2(x)| \leq \|\varphi\|_{\infty}|\mathcal{J}| \sigma^{\beta} \leq
  2\|\varphi\|_{\infty}A_{\sigma}(\beta)h_{\sigma}^{-1}\sigma^{\beta}$. Whenever
  $k \in \Lambda \cap \mathcal{J}^c$ and $\mu \in V_k$, it holds
  $|x - \mu| \geq |x - \mu_k| - |\mu - \mu_k| \geq \sigma
  A_{\sigma}(\beta)$. Therefore,
  $|r_3(x)| \lesssim \varphi(A_{\sigma}(\beta)) |\Lambda| \lesssim \sigma^{\beta
    + 1}$. With the same argument, \cref{pro:2} and Young's inequality we get
  $|r_4(x)| \lesssim \|\chi_{\sigma}*f_0\|_{\infty} \varphi(2A_{\sigma}(\beta))
  |\Lambda| \leq \|\chi\|_1\|f_0\|_{\infty} \sigma^{\beta}$. Regarding $r_5$,
  we rewrite $\Lambda^c = \Lambda_1^c\cup \Lambda_2^c$, with
  $\Lambda_1^c := \Set{k \in \mathbb Z \given |u_k| \leq \sigma^{\beta}}$ and
  $\Lambda_2^c := \Set{k \in \mathbb Z \given |\mu_k| > \sigma^{-2\beta/p} +
    \sigma \sqrt{2(\beta + 1)\log \sigma^{-1}}}$. Then,
  \begin{align}
    \notag
    |r_5(x)|
    &\leq
      \sum_{k\in \Lambda_1^c}|u_k|\, \varphi\left(\frac{x -
      \mu_k}{\sigma}\right) + \sum_{k\in \Lambda_2^c}|u_k|\,
      \varphi\left(\frac{x - \mu_k}{\sigma}\right)\\
    \label{eq:20}
    &\leq \sigma^{\beta}\sup_{x\in \mathbb R}\sum_{k\in \mathbb Z}
      \varphi\left(\frac{x - \mu_k}{\sigma}\right) + \sum_{k\in
      \Lambda_2^c}|u_k|\, \varphi\left(\frac{x - \mu_k}{\sigma}\right).
  \end{align}
  The first term of the rhs of \cref{eq:20} is bounded by a multiple constant of
  $h_{\sigma}^{-1}\sigma^{\beta}$, with the same argument as in the proof of
  \cref{pro:4}. By definition of $\Lambda_2^c$,
  $|x - \mu_k| \geq \sigma \sqrt{2(\beta + 1)\log \sigma^{-1}}$ when
  $k\in \Lambda^c_2$ and $|x| \leq \sigma^{-2\beta/p}$. This implies, together
  with \cref{pro:2} and Young's inequality, that the second term of the rhs of
  \cref{eq:20} is bounded by a constant multiple of
  $\sigma^{\beta + 1}\sum_{k\in\mathbb Z}|u_k| \lesssim
  \|\chi_{\sigma}*f_0\|_1\sigma^{\beta} \leq \|\chi\|_1\|f_0\|_1\sigma^{\beta}$
  for all $|x| \leq \sigma^{-2\beta/p}$. Finally, we have the trivial bound
  $|r_6(x)| \leq \|\varphi\|_{\infty}|M|(V^c) \leq
  \|\varphi\|_{\infty}\sigma^{\beta}$.
\end{proof}

\subsection{Construction of the approximation under location-scale and hybrid
  location-scale mixtures}
\label{sec:locat-scale-mixt-1}

Let $\sigma_0 := 1$ and define recursively $\sigma_{j+1} := \sigma_j/2$ for any
$j\geq 0$. Let $\Delta_0 := f_0 - \chi_{\sigma_0}*f_0$, and define recursively
$\Delta_{j+1} := \Delta_j - \chi_{\sigma_{j+1}}*\Delta_j$, for any $j\geq 0$.

The general idea of the construction is that
$|\Delta_j| \lesssim \sigma_j^\beta$, as shown in \cref{pro:5} in appendix, and
that similarly to wavelet decomposition, we approximate a function $f_0$
H\"older $\beta$ by
\begin{align*}
  f_1 := K_0(\chi_{\sigma_0}*f_0) +
  \sum_{j=1}^JK_j(\chi_{\sigma_j}*\Delta_{j-1}).
\end{align*}
where $J \geq 1$ is a large enough integer,
$h_J\sqrt{J} := 2\pi / \sqrt{\beta \log 2}$, and $K_j := K_{h_J,\sigma_j}$. By
induction, we get that
$\Delta_j = \Delta_0 - \sum_{l=0}^{j-1} \chi_{\sigma_{l+1}}*\Delta_l$. It
follows,
\begin{align*}
  f_1 - f_0
  &= K_0(\chi_{\sigma_0} * f_0) - f_0 +
    \sum_{j=1}^JK_j(\chi_{\sigma_j}*\Delta_{j-1})\\
  &= \Delta_J +
    K_0(\chi_{\sigma_0}*f_0) - \chi_{\sigma_0}*f_0+
    \sum_{j=1}^J\left[K_j(\chi_{\sigma_j}*\Delta_{j-1}) -
    \chi_{\sigma_j}* \Delta_{j-1}\right].
\end{align*}
Therefore, from \cref{lem:1,pro:5} and Young's inequality, the error of
approximating $f_0$ by $f_1$ is
\begin{multline*}
  |f_1(x) - f_0(x)|\\
  \begin{aligned}
    &\leq |\Delta_J| + |K_0(\chi_{\sigma_0}*f_0) - \chi_{\sigma_0}*f_0| +
    \sum_{j=1}^J|K_j(\chi_{\sigma_j}*\Delta_{j-1})
    - \chi_{\sigma_j}* \Delta_{j-1}|\\
    &\lesssim \|f_0\|_{\mathtt C^{\beta}} \sigma_J^{\beta} +
    \|\chi_{\sigma_0}*f_0\|_1 \sigma_0^{-1}e^{-4\pi^2/h_J^2} + e^{-4\pi^2/h_J^2}
    \sum_{j=1}^J\|\chi_{\sigma_j}*\Delta_{j-1}\|\sigma_j^{-1}\\
    &\lesssim \|f_0\|_{\mathtt C^{\beta}}\sigma_J^{\beta} + \|f\|_1
    e^{-4\pi^2/h_J^2}
    + \|f_0\|_1 e^{-4\pi^2/h_J^2}\sum_{j=1}^J2^j\\
    &\lesssim \|f_0\|_{\mathtt C^{\beta}}\sigma_J^{\beta} + \|f_0\|_1(1 +
    2^J)e^{-4\pi^2/h_J^2} \lesssim \sigma_J^{\beta}.
  \end{aligned}
\end{multline*}

The reason for considering different scale parameters in the construction, is to
deal with fat tail, the heuristic being that in the tail we do not require as
precise an approximation as in the center. In particular small values of $j$
will be used to estimate the function far off in the tails. To formalize this,
we define $\zeta_{j} := 2^{(J - j) (2\beta / p)}$, and
$A_{j} := [-\zeta_{j},\zeta_{j}]$, for all $j=0,\dots J$. We also define
$I_J = [-1,1]$, and for all $j=0,\dots,J - 1$ we set
$I_j := A_j \backslash A_{j+1}$. Notice that by definition of $K_j$, we can
write,
\begin{gather*}
  K_0(\chi_{\sigma_0}*f_0)(x)
  := \sum_{k\in \mathbb Z} u_{0k}\, \varphi((x - h_J\sigma_0 k)/\sigma_0)\\
  K_j(\chi_{\sigma_j}*\Delta_{j-1})(x) := \sum_{k\in \mathbb Z} u_{jk}\,
  \varphi((x - h_J\sigma_j k)/\sigma_j),\quad \forall j\geq 1.
\end{gather*}
To ease notation, we define $\mu_{jk} := h_J \sigma_j k$ for all $j\geq 0$ and
all $k\in \mathbb Z$. In the sequel we shall need the following subset of
indexes,
\begin{align*}
  \Lambda := \Set*{(j,k) \in \Set{0,\dots,J} \times \mathbb Z \given
  |u_{jk}| > \sigma_J^{\beta},\quad
  |\mu_{jk}| \leq \zeta_j + \sqrt{2(\beta + 1)\log \sigma_J^{-1}}}.
\end{align*}

We prove below that we can approximate $f_1$ by a finite mixture corresponding
to retaining only the components associated to indices in $\Lambda$ and that we
can bound the cardinality of $\Lambda $ by $O( J \log J \sigma_J^{-2\beta/p})$

To any $(j,k) \in \Lambda$ we associate
$U_j := \Set{\sigma \given \sigma_j \leq \sigma \leq \sigma_j(1+
  \sigma_J^{\beta}) }$,
$V_{jk} := \Set{\mu \given |\mu - \mu_{jk}| \leq \sigma_j \sigma_J^{\beta} }$
and $W_{jk} := U_j \times V_{jk}$. We denote by $\mathcal{M}$ the set of signed
measures $M$ on $(0,\infty)\times \mathbb R$ such that
$|M(W_{jk}) - u_{jk}| \leq \sigma_J^{\beta}$ for all $(j,k)\in \Lambda$, and
$|M|(W^c) \leq \sigma_J^{\beta}$, where $W^c$ is the relative complement of the
union of all $W_{jk}$ for $(j,k) \in \Lambda$. For any $M\in \mathcal{M}$, we
write $$f_M(x) := \int \varphi((x - \mu)/\sigma)\, dM(\sigma,\mu).$$

In \cref{pro:6} we control the cardinality of $\Lambda$ while in \cref{pro:3} we
control the error between $f_M$ and $f_1$ on the decreasing sequence of
intervals $[-\zeta_j, \zeta_j]$. \Cref{pro:8} provides a crude uniform upper
bound on $f_M $ and $f_0$.

\begin{proposition}
  \label{pro:6}
  There is a constant $C > 0$ depending only on $f_0$ and $Q_0$ such that
  $|\Lambda| \leq C[ \sigma_J^{-(\beta+1)} \wedge (J \log J)
  \sigma_J^{-2\beta/p}]$ if $p\leq 2\beta$, and
  $|\Lambda| \leq C(J \log J) \sigma_J^{-1}$ if $p>2\beta$.
\end{proposition}
\begin{proof}
  First notice that because of \cref{pro:2,pro:5}, we always have the bound
  \begin{equation}
    \label{eq:19}
    4\|f_0\|_1\sigma_J^{-1} \geq 2\|f_0\|_1\sum_{j=0}^J\sigma_j^{-1} \geq
    \sum_{j=0}^J\sum_{k\in \mathbb Z}|u_{jk}| \geq \sum_{(j,k)\in
      \Lambda}|u_{jk}| \geq \sigma_J^{\beta}|\Lambda|.
  \end{equation}
  If $p \leq 2\beta$, we define $B := \sqrt{2(\beta +1)\log 2}$, so that
  $\sqrt{2(\beta + 1)\log\sigma_J^{-1}} = B\sqrt{J}$.  Now consider those
  indexes $j$ with $\zeta_j \leq B\sqrt{J}$. An elementary computation shows
  that there are at most $\lesssim \log J$ such indexes. Therefore, recalling
  that there is a separation of $h_J\sigma_j$ between two consecutive $\mu_{jk}$
  and that there are at most $J$ indexes $j$ with $\zeta_j > B\sqrt{J}$
  \begin{align}
    \notag
    |\Lambda|
    &\lesssim
      \sum_{j=0}^J \frac{4\zeta_j}{h_J\sigma_j} + \log J
      \times\frac{2B\sqrt{J}}{h_J \sigma_J}\\
    \label{eq:8}
    &\leq 4h_J^{-1}\sigma_J^{-2\beta/p}\sum_{j=0}^J2^{-j(\frac{2\beta}{p} - 1)}
      + 2B(\sqrt{J}\log{J}) h_J^{-1}\sigma_J^{-1}.
  \end{align}
  Because $h_J\sqrt{J} \lesssim 1$ by definition, and because $p\leq 2\beta$,
  the result follows from the last equation and \cref{eq:19}. If $p > 2\beta$,
  the reasoning is the same as in the first part, but we can rewrite in this
  situation the \cref{eq:8} as
  \begin{align*}
    |\Lambda| \leq 4h_J^{-1}\sigma_J^{-1}\sum_{j=0}^J2^{(j - J)(1 -
    \frac{2\beta}{p})} + 2B(\sqrt{J}\log J)h_J^{-1}\sigma_J^{-1}.
  \end{align*}
  Since $p > 2\beta$, the conclusion is immediate.
\end{proof}

\begin{proposition}
  \label{pro:8}
  For all $x\in \mathbb R$, all $J > 0$ large enough and all $M\in \mathcal{M}$,
  it holds $|f_M(x) - f_0(x)| \lesssim J^{3/2}$.
\end{proposition}
\begin{proof}
  Let
  $\mathcal{I} \equiv \mathcal{I}(x) := \Set{(j,k) \in \Set{0,\dots,J}\times
    \mathbb Z \given |x - \mu_{jk}| \leq 2\sigma_j}$. Then the proof is almost
  identical to \cref{pro:9}. It suffices to notice that
  \begin{itemize}
    \item $|M|(W_{jk}) \leq |M(W_{jk}) - u_{jk}| + |u_{jk}|$ is always bounded
    above by a constant, because of the definition of $\mathcal{M}$, of
    \cref{pro:2,pro:5}.
    \item $|x - \mu| / \sigma \geq (1/4) |x - \mu_{jk}| / \sigma_j$ whenever
    $(\sigma,\mu) \in W_{jk}$ and $(j,k)\in \Lambda \cap \mathcal{I}^c$, as soon
    as $J$ is large enough.
    \item $|\mathcal{I}| \leq 5Jh_J^{-1}$ for $J\geq 1$. \qedhere
  \end{itemize}
\end{proof}

\begin{proposition}
  \label{pro:3}
  If $f_0 \in \mathcal C_\beta$, for all $J > 0$ large enough, all
  $0\leq j \leq J$, all $x\in [-\zeta_j,\zeta_j]$ and all $M\in \mathcal{M}$, it
  holds $|f_M(x) - f_0(x)| \lesssim J^{3/2}\sigma_j^{\beta}$.
\end{proposition}

The proof of \cref{pro:3} is given in \cref{pr:locationscale}.

\section{Proof of \cref{thm:bigthm}}
\label{sec:bigthm}

As mentioned earlier, the proof of \cref{thm:bigthm} boils down to verifying
conditions \eqref{eq:KLcond}, \eqref{eq:sievecond} and \eqref{eq:entropycond}
for the three types of priors.

\subsection{Case of the location mixture}
\label{sec:location}

\subsubsection{Kullback-Leibler condition for location mixtures}
\label{sec:kullb-leibl-cond-2}

In this Section we verify condition \eqref{eq:KLcond} in the case of the
location mixture prior, using the results of \cref{sec:location-mixtures-1}

\par By Chebychev inequality, we have
$Q_0[-\sigma^{-2\beta/p},\sigma^{2\beta/p}]^c \leq
\sigma^{2\beta}Q_0|X|^p$. Then by bringing together results from
\cref{pro:9,pro:10}, we can find a constant $C > 0$ such that for all
$M \in \mathcal{M}_{\sigma}$
\begin{align*}
  \int |f_{M,\sigma}(x) - f_0(x)|^2\, dQ_0(x)
  &\leq
    \sup_{|x| > \sigma^{-2\beta/p}}|f_{M,\sigma}(x) - f_0(x)|^2\,
    Q_0[-\sigma^{2\beta/p},\sigma^{2\beta/p}]^c\\
  &\quad+
    \sup_{|x| \leq \sigma^{-2\beta/p}}|f_{M,\sigma}(x) - f_0(x)|^2\\
  &\leq C\sigma^{2\beta}(\log \sigma^{-1})^2.
\end{align*}

\par By \cref{eq:17}, we have
$G_{\sigma}(U_{\sigma}) \gtrsim
\sigma^{-b_3}\sigma^{b_4\beta}\exp(-a_3/\sigma)$. Moreover, there is a
separation of $h_{\sigma}\sigma$ between two consecutive $\mu_{k}$ and
$h_{\sigma}\sigma \ll \sigma$, thus all the $V_{k}$ with $k \in \Lambda$ are
disjoint. By assumptions on $G_{\mu}$ (see \cref{eq:18}),
$\alpha_k := \overline{\alpha}G_{\mu}(V_k) \gtrsim \sigma^{b_5(\beta +1)}(1+
|\mu_k|)^{-b_6}$ for all $k \in \Lambda$. We also define
$\alpha^c := \alpha(V^c)$. For $\sigma$ small enough, there is a constant
$C' > 0$ not depending on $\sigma$ such that $\alpha^c > C'$. Moreover, since
$\alpha$ has finite variation we can assume without loss of generality that
$C' \leq \alpha^c \leq 1$, otherwise we split $V^c$ into disjoint parts, each of
them having $\alpha$-measure smaller than one. With
$\epsilon_n^2 := C\sigma^{2\beta}(\log \sigma^{-1})^2$, using that
$\Gamma(\alpha) \leq 2\alpha^{\alpha - 1}$ for $\alpha \leq 1$, it follows the
lower bound
\begin{align*}
  \notag
  \Pi(\mathtt{KL}(f_0,\epsilon_n))
  &\geq
    G_{\sigma}(U_{\sigma})\Pi_{\alpha}(\mathcal{M}_{\sigma})
    \gtrsim
    \sigma^{-b_3 + b_4\beta}e^{-a_3\sigma^{-1}}
    \frac{\sigma^{\beta}}{3e\Gamma(\alpha^c)}
    \prod_{k \in \Lambda}\left(
    \frac{\sigma^{\beta}e^{-2|u_k|}}{3e\Gamma(\alpha_k)} \right)\\
  \notag
  &\gtrsim \exp\left\{-K|\Lambda| \log \sigma^{-1} - a_3\sigma^{-1} -
    2\sum_{k\in \Lambda}|u_k| - \sum_{k\in \Lambda}\log \frac{1}{\alpha_k}
    \right\}\\
  &\gtrsim
    \exp\left\{-K |\Lambda|\log \sigma^{-1}
    - K\sigma^{-1}
    - \sum_{k\in \Lambda}\log \frac{1}{\alpha_k} \right\},
\end{align*}
for a generic constant $K > 0$. From the definition of $\alpha_k$, it holds
\begin{equation*}
  \sum_{k\in \Lambda}\log \frac{1}{\alpha_k}
  \lesssim
    |\Lambda|\log\sigma^{-1}
    + \sum_{k\in \Lambda}\log\left(1 +|\mu_k|\right),
\end{equation*}
when $\sigma$ is small enough. Also,
\begin{align*}
  \sum_{k\in \Lambda}\log\left(1 +|\mu_k|\right)
  &= \sum_{k\in\Lambda}\log \left( 1 + |\mu_{k}|\right)
    \Ind\Set{|\mu_{k}| \leq 1}
    + \sum_{k\in \Lambda}\log \left( 1 + |\mu_{k}|\right)
    \Ind\Set{|\mu_{k}| > 1}\\
  &\leq
    |\Set{k\in \Lambda \given |\mu_{k}| \leq 1}|
    + |\Lambda| \log 2
    + \sum_{k\in \Lambda}\log|\mu_{k}|\\
  &\leq 2h_{\sigma}^{-1}\sigma^{-1} + 4|\Lambda|\frac{2\beta}{p}\log \sigma^{-1}
    \lesssim |\Lambda|\log \sigma^{-1}+\sigma^{-1}
\end{align*}
Because $|\Lambda| > \sigma^{-1}$ for $\sigma$ small enough, it follows from all
of the above the existence of a constant $K' > 0$, depending only on $f$,
$\varphi$ and $\Pi$, such that
\begin{align*}
  \Pi(\mathtt{KL}(f_0,\epsilon_n))
  \geq \exp \left\{ - K'|\Lambda| \log \sigma^{-1}\right\}.
\end{align*}
Then for an appropriate constant $C''' > 0$, as a consequence of \cref{pro:13},
we can have $\Pi(\mathtt{KL}(f_0,\epsilon_n)) \geq e^{-c_2n\epsilon_n^2}$ if
\begin{align*}
  \epsilon_n^2 =
  \begin{cases}
    C''' n^{-2\beta/(3\beta + 1)} (\log n)^{2-2\beta/(3\beta+1)} & 0 < p \leq
    2,\\
    C''' n^{-2\beta/(2\beta + 1 + 2\beta/p)} (\log
    n)^{2-3\beta/(2\beta+1+2\beta/p)} & p > 2.
  \end{cases}
\end{align*}

\subsubsection{Sieve construction for location mixtures}
\label{sec:construction-tests-1}

We construct the following sequence of subsets of $\mathcal{F}$, also called
\textit{a sieve}. With the notation
$f_{M,\sigma}(x) := \int \varphi((x-\mu)/\sigma)\, dM(\mu)$,
\begin{gather*}
  \mathcal{F}_n(H,\epsilon) := \Set*{f = f_{M,\sigma} \given
  \begin{array}{l}
    M = \sum_{i=1}^{\infty}u_i\delta_{\mu_i},\quad n^{-1/b_2} < \sigma \leq
    n^{1/b_1}\\
    \sum_{i=1}^{\infty}|u_i| \leq n,\quad \sum_{i=1}^{\infty}|u_i|\Ind\Set{|u_i|
    \leq n^{-1}} \leq \epsilon\\
    |\Set{ i \given |u_i| > n^{-1} }| \leq H n \epsilon^2 /\log n
  \end{array}
  }.
\end{gather*}
The next two lemmas show that $\mathcal{F}_n(H,\epsilon)$ defined as above
satisfies all the condition stated in \cref{eq:sievecond,eq:entropycond} if $H$
and $\delta$ are chosen small enough.

\begin{lemma}
  \label{lem:3}
  Let $\mathbf x = (x_1,\dots,x_n) \in \mathbb R^n$ be arbitrary and $d_n$ be
  the empirical $L^2$-distance associated with $\mathbf x$. Then for any
  $n^{-1/2} < \epsilon_n \leq 1$, $0 < H \leq 1$ and $n$ sufficiently large
  there is a constant $C > 0$ not depending on $n$ such that
  $\log N(\epsilon_n, \mathcal{F}_n(H,\epsilon_n), d_n) \leq C H n\epsilon_n^2$.
\end{lemma}
\begin{proof}
  We write $\mathcal{F}_n \equiv \mathcal{F}_n(H,\epsilon_n)$ to ease notations.
  The proof is based on arguments from \citet{ShenTokdarGhosal2013}, it uses the
  fact that the covering number $N(\epsilon_n,\mathcal{F}_n,d_n)$ is the minimal
  cardinality of an $\epsilon_n$-net over $(\mathcal{F}_n,d_n)$. We recall that
  $(\mathcal{F}_n,d_n)$ has $\epsilon_n$-net $\mathcal F_{n,\epsilon}$, if for
  any $f \in \mathcal{F}_n$ we have $m\in \mathcal F_{n,\epsilon}$ such that
  $d_n(f,m) < \epsilon_n$. Let
  $S_n := \cup_{i=1}^n\Set{x \given |x - x_i| \leq n^{1/b_1}\sqrt{6 \log n}}$,
  $R_n := \Set{\mu \in \mathbb R \given \mu = k / n^{3/2+ 1/b_2},\ k\in \mathbb
    Z,\ \mu \in S_n}$ and,
  \begin{equation*}
    \mathcal F_{n,\epsilon} := \Set*{ f =
      \textstyle\sum_{i\in \mathcal{I}} u_i\,
      \varphi\left( \frac{\cdot - \mu_i}{\sigma} \right) \given
      \begin{array}{l}
        |\mathcal{I}| \leq Hn\epsilon_n^2/\log n,\
        n^{-1/b_2} \leq  \sigma \leq n^{1/b_1}\\
        \forall i \in \mathcal{I}:\, |u_i| \leq n,\ \mu_i \in R_n\\
        u_i = k n^{-3/2}H^{-1},\ k\in \mathbb Z,\\
        \sigma = k / n^{3/2+1/b_2},\ k\in \mathbb N,
      \end{array}
    }.
  \end{equation*}
  We claim that there is a constant $\delta > 0$ such that
  $\mathcal F_{n,\epsilon}$ is a $\delta \epsilon$-net over
  $(\mathcal{F}_n,d_n)$. Indeed, let $f \in \mathcal{F}_n$ be arbitrary, so that
  $f = \sum_{i=1}^{\infty}u_i\, \varphi((\cdot - \mu_i)/\sigma)$. We define
  $\mathcal{J} := \mathbb N \cup \Set{\infty}$,
  $\mathcal{K} := \Set{i \given |u_i| > n^{-1}}$, and
  $\mathcal{L} := \Set{i \given \mu_i \in S_n}$. Now choose
  $\mathcal{I} = \mathcal{J} \cap \mathcal{K} \cap \mathcal{L}$, and notice that
  $|\mathcal{I}| \leq |\mathcal{K}| \leq Hn\epsilon_n^2 / \log n$. Hence we can
  pick a $m \in \mathcal F_{n,\epsilon}$ with
  $m(x) = \sum_{i\in \mathcal{I}}u_i'\, \varphi((x -
  \mu_i')/\sigma')$. Moreover, for any $j=1,\dots,n$
  \begin{multline*}
    |f(x_j) - m(x_j)|
    \leq
    \sum_{\mathcal{J}\cap \mathcal{K} \cap \mathcal{L}^c}|u_i| \varphi((x_j -
    \mu_i)/\sigma) + \sum_{\mathcal{J} \cap \mathcal{K}^c}|u_i| \varphi((x_j -
    \mu_i)/\sigma)\\
    + \sum_{i \in \mathcal{I}}|u_i|| \varphi((x_j - \mu_i)/\sigma) -
    \varphi((x_j - \mu_i')/\sigma')|\\
    + \sum_{i \in \mathcal{I}} |u_i - u_i'|\, \varphi((x_j -
    \mu_i')/\sigma').
  \end{multline*}
  The fourth term in the rhs of the last equation is bounded above by
  $\epsilon_n$. Regarding the third term, for any $i \in \mathcal{L}^c$ we have
  $|x_j - \mu_i|/\sigma > \sqrt{6\log n}$ for all $j=1,\dots,n$. Then the third
  term is bounded by
  $|\mathcal{K}| n \varphi(\sqrt{6 \log n}) \leq H n\epsilon_n^2 n^{-2} / \log n
  \leq \epsilon_{n}$. Since we can always choose $m \in \mathcal F_{n,\epsilon}$
  with $|u_i - u_i'| \leq n^{-3/2}H^{-1}$ for all $i\in \mathcal{I}$,
  $|\mu_i - \mu_i'| \leq n^{-3/2-1/b_2}$ for all $i\in \mathcal{I}$, and
  $|\sigma - \sigma'| \leq n^{-3/2-1/b_2}$, it follows from \cref{pro:1}
  \begin{multline*}
    |f(x_j) - m(x_j)|\\
    \begin{aligned}
      &\leq 2\epsilon_n + \sum_{i \in \mathcal{I}} |u_i - u_i'| + \sum_{i \in
        \mathcal{I}}|u_i|| \varphi((x_j - \mu_i)/\sigma) - \varphi((x_j -
      \mu_i')/\sigma')|\\
      &\leq 2\epsilon_n + \sum_{i\in \mathcal{I}}|u_i - u_i'| + 4\sum_{i\in
        I}|u_i|\frac{|\sigma_i - \sigma_i'|}{\sigma_i \vee \sigma_i'} +
      \sum_{i\in \mathcal{I}} |u_i|\frac{|\mu_i - \mu_i'|}{\sigma_i\vee
        \sigma_i'} \leq 8\epsilon_n,
    \end{aligned}
  \end{multline*}
  for all $j=1,\dots,n$. Therefore $d_n(f,m) \leq 8\epsilon_n$, and the claim is
  proved with $\delta := 8$. To finish the proof, it suffices to compute the
  cardinality of $\mathcal F_{n,\epsilon}$. A straightforward computation shows
  that
  $|R_n| \leq n^{5/2+1/b_1 + 1/b_2} \sqrt{6\log n} \leq n^{4 + 1/b_1 + 1/b_2}$
  for all $n\geq 1$, then
  \begin{align*}
    \log N(c_3\epsilon_n, \mathcal{F}_n, d_n)
    &\leq |\mathcal{I}| \log\left(
      \frac{n}{n^{-3/2}}\times n^{4+1/b_1+1/b_2} \right)
      + \log \left(\frac{n^{1/b_1}}{n^{-3/2-1/b_2}}\right)\\
    &\leq H\left(\frac{11}{2} +\frac{2}{b_1} + \frac{2}{b_2}\right)
      n\epsilon_n^2,
  \end{align*}
  where the last line holds when $n$ becomes large enough. Then the lemma is
  proved with $C := (11/2 + 2/b_1 + 2/b_2)/64$.
\end{proof}

\begin{lemma}
  \label{lem:5}
  Assume that there is $n_0 \in \mathbb N$, and $0 < \gamma_1 \leq \gamma_2 < 1$
  such that $n^{-\gamma_2/2} \leq \epsilon_n \leq n^{-\gamma_1/2}$ for all
  $n \geq n_0$. Then
  $\Pi(\mathcal{F}_n(H,\epsilon_n)^c) \lesssim \exp(- \frac H4(1 - \gamma_2)
  n\epsilon_n^2)$ for all $n \geq n_0$.
\end{lemma}
\begin{proof}
  We use the fact that $M \sim \Pi_{\alpha}$ is almost surely purely-atomic
  \citep{Kingman1992}). Then from the definition of $\mathcal{F}_n$ it follows
  \begin{multline*}
    \Pi(\mathcal{F}_n^c) \leq G_{\sigma}(\sigma \leq n^{-1/b_2}) +
    G_{\sigma}(\sigma > n^{1/b_1})
    + \Pi_{\alpha}\Big(\textstyle\sum_{i=1}^{\infty}|u_i| > n\Big)\\
    +\Pi_{\alpha}\Big( \textstyle\sum_{i=1}^{\infty}|u_i|\Ind\Set{|u_i|\leq
        n^{-1}} > \epsilon_n \Big) +\\ \Pi_{\alpha}\Big(|\Set{i \given
        |u_i|> n^{-1}}| > Hn\epsilon_n^2/\log n\Big).
  \end{multline*}
  We bound each of the term as follows. By assumption
  $G_{\sigma}(\sigma \leq n^{-1/b_2}) \lesssim e^{-a_2 n}$ and
  $G_{\sigma}(\sigma > n^{1/b_1}) \lesssim e^{-a_1 n}$. Notice that
  $\sum_{i=1}^{\infty}|u_i| = |M|$, where $|M|$ denote the total variation of
  the measure $M$. Since by definition we have $M \overset{d}{=} M_1 - M_2$,
  with $M_1,M_2$ independent Gamma random measures with same base measure
  $\alpha(\cdot)$, it follows that $|Q|$ has the distribution of a Gamma random
  variable with shape parameter $2\overline{\alpha}$. Then by Markov's
  inequality,
  \begin{align*}
    \Pi_{\alpha}\Big(\textstyle\sum_{i=1}^{\infty}|u_i| > n\Big)
    &= \Pi_{\alpha}\left(e^{\frac 12 |M|} > e^{\frac 12 n}\right)
      \leq 2^{2\overline{\alpha}}e^{- \frac 12 n}.
  \end{align*}
  Also, by the superposition theorem \citep[section~2]{Kingman1992}, for any
  $M \sim \Pi_{\alpha}$ we have $M \overset{d}{=} M_3 + M_4$, where $M_3$ and
  $M_4$ are independent random measures with total variation $|M_3|$ and $|M_4|$
  having Laplace transforms (for all $t\in \mathbb R$ for which the integrals in
  the expressions converge)
  \begin{gather*}
    Ee^{t |M_3|} := \exp\left\{ 2\overline{\alpha}\int_{1/n}^{\infty}(e^{tx}
      - 1)x^{-1}e^{-x}\, dx \right\},\\
    Ee^{t|M_4|} := \exp\left\{ 2\overline{\alpha}\int_0^{1/n}(e^{tx} -
      1)x^{-1}e^{-x}\, dx \right\}.
  \end{gather*}
  $M_3$ and $M_4$ are almost-surely purely atomic, $M_3$ has only jumps greater
  than $1/n$ (almost surely) which number is distributed according to a Poisson
  distribution with intensity $2\overline{\alpha}E_1(n^{-1})$, where $E_1$
  denotes the exponential integral $E_1$ function:
  $E_1(x) = \int_x^\infty \frac{ e^{-t}}{ t } dt$. Likewise, $M_4$ has only
  jumps smaller or equal to $1/n$ (almost-surely) which number is almost-surely
  infinite. Recalling that $E_1(x) = \gamma + \log(1/x)+ o(1) $ for $x$ small,
  it holds
  $2\overline{\alpha}\gamma \leq 2\overline{\alpha} E_1(1/n) \leq
  6\overline{\alpha} \log n \leq x_n $ for $n$ sufficiently large, with
  $x_n := Hn\epsilon_n^2/\log n$. Thus using Chernoff's bound on Poisson
  distribution, we get
  \begin{align*}
    \Pi_{\alpha}\Big( |\Set{i \given |u_i|> n^{-1} }| >
    Hn\epsilon_n^2/\log n\Big)
    &\leq
      e^{-2\overline{\alpha}E_1(1/n)}
      \frac{(e2\overline{\alpha}E_1(1/n))^{x_n}}{x_{n}^{x_n}}\\
    &\leq \exp\left\{- \frac 12 x_n \log x_n \right\}.
  \end{align*}
  But, $\log x_n = \log n + \log H - 2\log\epsilon_n^{-1} - \log \log n \geq (1
  - \gamma_2)\log n + \log H - \log \log n \geq \frac 12 (1 - \gamma_2)\log n$
  for large $n$. Therefore, as $n\rightarrow \infty$
  \begin{equation*}
    \Pi_{\alpha}\Big( |\Set{i \given |u_i|> n^{-1}}| >
      Hn\epsilon_n^2/\log n\Big) \leq \exp \left\{ -\frac{H}{4}(1-\gamma_2)
      n\epsilon_n^2 \right\}.
  \end{equation*}
  Finally, we use again Markov's inequality to get
  \begin{multline*}
    \Pi_{\alpha}\Big(
    \textstyle\sum_{i=1}^{\infty}|u_i|\Ind\Set{|u_i|\leq n^{-1}} >
    \epsilon_n\Big)\\
    \begin{aligned}
      &= \Pi_{\alpha}\Big( e^{n\epsilon_n |M_4|} > e^{n\epsilon_n^2}\Big)\\
      &\leq e^{-n\epsilon_n^2}\exp\left\{2\overline{\alpha}\int_0^{1/n}
        (e^{n\epsilon_n x }-1)x^{-1}e^{-x}\,dx \right\}.
    \end{aligned}
  \end{multline*}
  But for $x \in (0,1/n)$, we have
  $e^{n\epsilon_n x} - 1 \leq n(e^{n\epsilon_n\delta_n } - 1) x$, thus the
  integral in the previous expression is bounded by
  $2\overline{\alpha}(e^{\epsilon_n} - 1)$, which is in turn bounded by
  $2\overline{\alpha}(e - 1)$ because $\epsilon_n \leq 1$ if $n\geq n_0$.
\end{proof}

\subsection{Case of the location-scale mixture}
\label{sec:location-scale}

\subsubsection{Kullback-Leibler condition}
\label{sec:kullb-leibl-cond}

\par By Chebychev inequality, we have
$Q_0[-\zeta_j,\zeta_j]^c \leq \zeta_j^{-p}Q_0|X|^p$. Therefore, bringing
together results from \cref{pro:8,pro:3},
\begin{multline*}
  \int |f_M(x) - f_0(x)|^2\, dQ_0(x)\\
  \begin{aligned}
    &= \sum_{j=0}^J \int_{I_j} |f_M(x) - f_0(x)|^2\, dQ_0(x)
    + \int_{A_0^c}|f_M(x) - f(x)|^2\, dQ_0(x)\\
    &\lesssim J^3 \sum_{j=0}^J \sigma_j^{2\beta} Q_0(I_j) + J^3 Q_{0}(A_0^c) .
  \end{aligned}
\end{multline*}
Then we can find a constant $C > 0$ such that
$\int |f_M(x) - f_0(x)|^2\, dQ_0(x) \leq C J^4\sigma_J^{2\beta}$ for all
$M \in \mathcal{M}$ and $J$ large enough.

\par By \cref{eq:17}, we have
$G_{\sigma}(U_j) \gtrsim
\sigma_j^{-b_3}\sigma^{b_4\beta}_J\exp(-a_3/\sigma_j)$ for all $j=0,\dots
J$. Moreover, there is a separation of $h_J\sigma_j$ between two consecutive
$\mu_{jk}$ and $h_J\sigma_j \ll \sigma_j$, thus all the $W_{jk}$ with
$(j,k) \in \Lambda$ are disjoint. By \cref{eq:18}, we have
$\alpha_{jk} := \overline{\alpha}G_{\sigma}(U_{j})G_{\mu}(V_{jk}) \gtrsim
\sigma_j^{b_5(\beta +1) + b_4\beta}\exp(-a_3/\sigma_j) (1+
|\mu_{jk}|)^{-b_6}$ for all $(j,k) \in \Lambda$. We also define
$\alpha^c := \alpha(W^c)$. For $J$ large enough, there is a constant $C' > 0$
not depending on $J$ such that $\alpha^c > C'$. Moreover, since $\alpha$ has
finite variation we can assume without loss of generality that
$C' \leq \alpha^c \leq 1$, otherwise we split $W^c$ into disjoint parts, each of
them having $\alpha$-measure smaller than one. With
$\epsilon_n^2 := CJ^4\sigma_J^{2\beta}$, using that
$\Gamma(\alpha) \leq 2\alpha^{\alpha - 1}$ for $\alpha \leq 1$ and $\mathcal{M}
\subset \mathtt KL(f_0,\epsilon_n)$, it follows the
lower bound
\begin{multline}
  \label{eq:9}
  \Pi(\mathtt{KL}(f_0,\epsilon_n)) \geq \frac{\sigma_J^{\beta}}{3e
    \Gamma(\alpha^c)} \prod_{(j,k) \in \Lambda}\left(
    \frac{\sigma_J^{\beta}e^{-2|u_{jk}|}}{3e\Gamma(\alpha_{jk})} \right)\\
  \begin{aligned}
    &\geq \frac{\sigma_J^{\beta}}{3e \Gamma(\alpha^c)} \prod_{(j,k) \in \Lambda}
    \exp \left\{ -2|u_{jk}| - \beta \log\sigma_J^{-1} + \log\frac{1}{6e} +
      (\alpha_{jk} -
      1)\log \alpha_{jk} \right\}\\
    &\geq \exp \left\{ - K J |\Lambda| -2
      \textstyle\sum_{(j,k)\in\Lambda}|u_{jk}| -
      \sum_{(j,k)\in \Lambda}\log \alpha_{jk}^{-1} \right\},
  \end{aligned}
\end{multline}
for a constant $K > 0$ depending only on $C$ and $\beta$. We now evaluate the
sums involved in the rhs of \cref{eq:9}. As before, be have that
$\sum_{(j,k)\in \Lambda}|u_{jk}| \leq 4\|f_0\|_1\sigma_J^{-1}$ (see for instance
the proof of \cref{pro:3}). Act as in \cref{sec:kullb-leibl-cond-2} to find that
\begin{align*}
  \sum_{(j,k)\in \Lambda} \log\alpha_{jk}^{-1}
  &\lesssim J|\Lambda| + J^{3/2}\sigma_J^{-1} + |\Lambda|\sigma_J^{-1}.
\end{align*}
The term proportional to $|\Lambda|\sigma_J^{-1}$ is entirely responsible for
the bad rates in location-scale mixtures, and the aim of the hybridation of next
section is to get rid of it. For a constant $K' > 0$,
\begin{align*}
  \Pi(\mathtt{KL}(f_0,\epsilon_n))
  &\geq \exp\left\{ -K'|\Lambda|\sigma_J^{-{1}} \right\}.
\end{align*}
Then for an appropriate constant $C' > 0$ we can have
$\Pi(\mathtt{KL}(f_0,\epsilon_n)) \geq e^{-c_2n\epsilon_n^2}$ if
\begin{equation*}
  \epsilon_n^2 =
  \begin{cases}
    C' [n^{-2\beta/(3\beta + 2)}(\log n)^{t_1} \wedge
    n^{-2\beta/(2\beta+1 + 2\beta/p)}(\log n)^{t_2} ],&\quad p\leq 2\beta,\\
    C' n^{-\beta/(\beta+1)}(\log n)^{t_3},&\quad p > 2\beta,
  \end{cases}
\end{equation*}
where $t_1 := 4-8\beta/(3\beta + 2)$ ,
$t_2:= 4-4\beta / (2\beta + 1 + 2\beta/p)$ and $t_3 := 4-2\beta/(\beta+1)$.

\subsubsection{Sieve construction}
\label{sec:construction-tests-2}

Using the notation $f_M(x) := \int \varphi((x-\mu)/\sigma)\,dM(\sigma,\mu)$, we
construct the following sieve.
\begin{equation}
  \label{eq:10}
  \mathcal{F}_n(H,\epsilon) :=
  \resizebox{0.75\textwidth}{!}{
    $\Set*{f = f_M
    \begin{array}{l}
      M = \sum_{i=1}^{\infty}u_i\delta_{\sigma_i,\mu_i},\quad
      \sum_{i=1}^{\infty}|u_i| \leq n,\\
      |\Set{ i \given |u_i| > n^{-1},\ n^{-1/b_2} < \sigma_i \leq n^{1/b_1} }|
      \leq H n \epsilon^2 /\log n,\\
      \sum_{i=1}^{\infty}|u_i|\Ind\Set{|u_i| \leq n^{-1}} \leq \epsilon,\\
      \sum_{i=1}^{\infty}|u_i|\Ind\Set{ \sigma_i \leq n^{-1/b_2}} \leq \epsilon,
      \ \sum_{i=1}^{\infty}|u_i|\Ind\Set{ \sigma_i > n^{1/b_1}} \leq \epsilon
    \end{array}
  }.$}
\end{equation}

\begin{lemma}
  \label{lem:6}
  Let $x = (x_1,\dots,x_n) \in \mathbb R^n$ be arbitrary and $d_n$ be the
  empirical $L^2$-distance associated with $x$. Then for any
  $n^{-1/2} < \epsilon_n \leq 1$, $0 < H \leq 1$ and $n$ sufficiently large
  there is a constant $C > 0$ not depending on $n$ such that
  $\log N(\epsilon_n, \mathcal{F}_n(H,\epsilon_n), d_n) \leq C H n\epsilon_n^2$.
\end{lemma}

The proof is almost identical to \cref{lem:3}, with the same constant $C > 0$.

\begin{lemma}
  \label{lem:7}
  Assume that there is $n_0 \in \mathbb N$, and $0 < \gamma_1 \leq \gamma_2 < 1$
  such that $n^{-\gamma_2/2} \leq \epsilon_n \leq n^{-\gamma_1/2}$ for all
  $n \geq n_0$. Then
  $\Pi(\mathcal{F}_n(H,\epsilon_n)^c) \lesssim \exp(- \frac H4(1 - \gamma_2)
  n\epsilon_n^2)$ for all $n \geq n_0$.
\end{lemma}
\begin{proof}
  We first write the estimate
  \begin{multline*}
    \Pi(\mathcal{F}_n^c) \leq \Pi_{\alpha}\Big(
    \textstyle\sum_{i=1}^{\infty}|u_i| > n\Big) +\Pi_{\alpha}\Big(
    \textstyle\sum_{i=1}^{\infty}|u_i| \Ind\Set{|u_i|\leq \delta} >
    \epsilon_n\Big)\\
    + \Pi_{\alpha}\Big(\textstyle\sum_{i=1}^{\infty}|u_i| \Ind\Set{ \sigma_i
      \leq n^{-1/b_2}} > \epsilon_n\Big) +
    \Pi_{\alpha}\Big(\textstyle\sum_{i=1}^{\infty}|u_i| \Ind\Set{ \sigma_i >
      n^{1/b_1}} > \epsilon_n\Big)\\
    +\Pi_{\alpha}\Big(|\Set{i \given |u_i|>\delta,\ n^{-1/b_2}< \sigma_i \leq
      n^{1/b_1}}| > Hn\epsilon_n^2/\log n\Big).
  \end{multline*}
  The first three terms in the rhs above obeys the same bounds as in the proof
  of \cref{lem:5}, using the same arguments. The last two term are bounded using
  the same trick, thus we simply bound the last term and left the other to the
  reader. Notice that the random variable
  $U := \sum_{i=1}^{\infty}|u_i| \Ind\Set{\sigma_i > n^{1/b_1}}$ has Gamma
  distribution with parameters $2\alpha(A_n),1$, with
  $A_n := \Set{(\sigma,\mu) \given \sigma > n^{1/b_1}}$. For $n$ large, by
  assumptions on $P_{\sigma}$, it holds $\alpha(A_n) \ll \epsilon_n$. Then by
  Chebychev inequality, for $n$ large enough
  \begin{align}
    \notag
    \Pi_{\alpha}\Big(\textstyle\sum_{i=1}^{\infty}|u_i| \Ind\Set{ \sigma_i >
    n^{1/b_1}} > \epsilon_n\Big)
    &\leq \Pr( U - EU > \epsilon_n - EU)\\
    \label{eq:11}
    &\leq \Pr(U -EU > \epsilon_n/2) \leq 16 \epsilon_n^{-2}\alpha(A_n)^2.
  \end{align}
  The conclusion follows from the assumptions on $G_{\sigma}$ which imply
  $\alpha(A_n) = \overline{\alpha} G_{\sigma}(\sigma > n^{1/b_1}) \lesssim
  \exp(-a_1 n)$.
\end{proof}

\subsection{Hybrid location-scale mixtures}
\label{sec:hybr-locat-scale-1}

Obviously, given the definition of hybrid mixtures (see
\cref{sec:hybr-locat-scale-1}), most of the proof is redundant with the
location-scale case, and in the sequel we deal only with the parts that differ.

\subsubsection{Kullback-leibler condition}
\label{sec:kullb-leibl-cond-3}

Let $\mathcal{M} \equiv \mathcal{M}(\beta,J,f,\Lambda)$ be the set of signed
measures constructed in \cref{sec:locat-scale-mixt-1}. For any integer $J > 0$
let $\Omega_J$ be the event
\begin{equation*}
  \Omega_J := \Set*{P_{\sigma} \given
    P_{\sigma}[2^{-j},2^{-j}(1+2^{-J\beta})] \geq
    2^{-J} \quad \forall \ 0 \leq j \leq J}.
\end{equation*}
Then with arguments and constant $C >0$ from \cref{sec:kullb-leibl-cond},
letting $\epsilon_n^2 := C J^4 \sigma_J^{2\beta}$, we have
\begin{equation*}
  \Pi(\mathtt{KL}(f_0,\epsilon_n)) \geq \Pi(\mathcal{M}) \geq \Pi(\mathcal{M} \mid
  \Omega_J) \Pi_{\sigma}(\Omega_J).
\end{equation*}
But by \cref{eq:6} we have
$\Pi_{\sigma}(\Omega_J) \gtrsim\exp(-a_6J^{b_7} 2^J)$ and on $\Omega_J$ it
holds
$\alpha(W_{jk}) = \overline{\alpha} P_{\sigma}(U_j)G_{\mu}(V_{jk}) \geq
\overline{\alpha}2^{-J}G_{\mu}(V_{jk})$ for all $(j,k) \in \Lambda$. Then act as
in \cref{eq:9} to find a constant $K > 0$ such that (recalling that
$\sigma_J = 2^{-J}$)
\begin{align*}
  \Pi(\mathtt{KL}(f_0,\epsilon_n))
  &\gtrsim \exp\left\{-K(J^{b_7} \vee J^{1/2})\sigma_J^{-1} - K J |\Lambda|
    \right\}.
\end{align*}
Because of \cref{pro:6} we can have
$\Pi(\mathtt{KL}(f_0,\epsilon_n)) \geq e^{-c_2n\epsilon_n^2}$ if for an
appropriate constant $C' > 0$
\begin{equation*}
  \epsilon_n^2 =
  \begin{cases}
    C'[n^{-2\beta/(3\beta+1)}(\log n)^{4-6\beta/(3\beta+1)} \wedge
    n^{-p/(p+1)}(\log n)^{4-p/(p+1)}]& p\leq 2\beta,\\
    C'n^{-2\beta/(2\beta + 1)}(\log n)^{4- 2\beta(4-b_7\vee
      3)/(2\beta+1)} & p > 2\beta.
  \end{cases}
\end{equation*}

\subsubsection{Sieve construction}
\label{sec:construction-tests-3}

We use the same sieve $\mathcal{F}_n(H,\epsilon)$ as in \cref{eq:10}. The
definition of $\mathcal{F}_n(j,\epsilon)$ is independent of $\Pi$ thus the
conclusion of \cref{lem:3} holds for hybrid location-scale mixtures. It remains
to show that $\Pi(\mathcal{F}_n(H,\epsilon)^c) \leq \exp(-2c_2n\epsilon_n^2)$,
which is the object of the next lemma.

\begin{lemma}
  \label{lem:9}
  Assume that there is $n_0 \in \mathbb N$, and $0 < \gamma_1 \leq \gamma_2 < 1$
  such that $n^{-\gamma_2/2} \leq \epsilon_n \leq n^{-\gamma_1/2}$ for all
  $n \geq n_0$. Then there is a constant a constant $\gamma_2 < \gamma < 1$ such
  that
  $\Pi(\mathcal{F}_n(H,\epsilon_n)^c) \lesssim \exp(- \frac H4(1 - \gamma)
  n\epsilon_n^2)$ for all $n \geq n_0$.
\end{lemma}
\begin{proof}
  We proceed as in the proof of \cref{lem:7}. Following the same steps, we
  deduce that it is sufficient to prove that
  \begin{gather*}
    \Pi_{\alpha}\Big(
      \textstyle\sum_{i=1}^{\infty}|u_i|\Ind\Set{\sigma_i > n^{1/b_1}} >
      \epsilon_n\Big) \lesssim e^{- 2c_2n},\\
    \Pi_{\alpha}\Big(
      \textstyle\sum_{i=1}^{\infty}|u_i|\Ind\Set{\sigma_i \leq n^{-1/b_2}} >
      \epsilon_n\Big) \lesssim e^{- 2c_2 n}.
  \end{gather*}
  Since the proofs are almost identical for the two previous conditions, we only
  prove the first and left the second to the reader. Notice that by \cref{eq:11}
  we have
  \begin{equation*}
    \Pi_{\alpha}\Big(
    \textstyle\sum_{i=1}^{\infty}|u_i|\Ind\Set{\sigma_i > n^{1/b_1}} >
    \epsilon_n \ \Big|\ P_{\sigma}\Big) \leq 16 \overline{\alpha}
    \epsilon_n^{-2}P_{\sigma}(\sigma > n^{1/b_1})^2.
  \end{equation*}
  Letting
  $\Omega := \Set{P_{\sigma} \given P_{\sigma}(\sigma > n^{1/b_1}) < \exp(-a_1 n
    /2)}$, with a slight abuse of notation, it follows from \cref{eq:4}
  \begin{multline*}
    \Pi_{\alpha}\Big( \textstyle\sum_{i=1}^{\infty}|u_i|\Ind\Set{\sigma_i >
      n^{1/b_1}} > \epsilon_n\Big)\\
    \begin{aligned}
      &\leq \Pi_{\alpha}\Big(
      \textstyle\sum_{i=1}^{\infty}|u_i|\Ind\Set{\sigma_i > n^{1/b_1}} >
      \epsilon_n\ \Big|\ \Omega\Big)
      + \Pi_{\sigma}(\Omega^c)\\
      &\lesssim \epsilon_n^{-2}\exp(-a_1n) + \exp(-a_4 n).&\qedhere
    \end{aligned}
  \end{multline*}
\end{proof}

\section{Proof of \cref{thm:1}}

The proof follows the same lines as \citet{GhosalVanDerVaartothers2007} with
additional cares. The first step consists on rewriting expectation of the
posterior distribution as follows. Let $(\phi_n(\cdot \mid \cdot))_{n\geq 0}$ be
a sequence of test functions such that for $n$ large enough
\begin{gather*}
  Q_0^n\left[P_0^n[\phi_n(\mathbf y \mid \mathbf x) \mid \mathbf x]\right]
  \lesssim N(\epsilon/18,\mathcal{F}_n,d_n)\exp\left(-\frac{n\epsilon_n^2}{2}
  \right),\\
  \sup_{\Set{f \given d_n(f,f_0)\geq 17\epsilon_n / 18}\cap \mathcal{F}_n}
  Q_0^n\left[P_f^n[1 - \phi_{n}(\mathbf y \mid \mathbf x)] \mid \mathbf x
  \right] \lesssim \exp\left(-\frac{n\epsilon_n^2}{2} \right).
\end{gather*}
The existence of such test functions is standard and follows for instance from
\citet[proposition~4]{Birge2006}, or \citet[section~7.7]{ghosal2007bis}. From
here, we bound the posterior distribution in a standard fashion,
\begin{multline*}
  Q_0^n\left[P_0^n[\Pi_{\mathbf x}(\Set{f \given d_n(f,f_0) > \epsilon_n} \mid
\mathbf y, \mathbf x) \mid \mathbf x] \right]
  \leq Q_0^n[P_0^n[\Pi_{\mathbf x}(\mathcal{F}_n^c\mid \mathbf y, \mathbf
  x) \mid \mathbf x]]\\
  +
  Q_0^n[P_0^n[\Pi(\Set{f \given d_n(f,f_0) > \epsilon_n} \cap \mathcal{F}_n \mid
  \mathbf y, \mathbf x) \mid \mathbf x]].
\end{multline*}
So that,
\begin{multline}
  \label{eq:1}
  Q_0^n\left[P_0^n[\Pi_{\mathbf x}(\Set{f \given d_n(f,f_0) > \epsilon_n} \mid
    \mathbf y, \mathbf x) \mid \mathbf x] \right] \leq Q_0^n[P_0^n[\Pi_{\mathbf
    x}(\mathcal{F}_n^c\mid \mathbf y, \mathbf x) \mid \mathbf x ]]\\
  \begin{aligned}
    &\quad+ Q_0^n\left[ P_0^n\left[\phi_n(\mathbf y\mid \mathbf
        x)\Pi_{\mathbf x}(\Set{f \given d_n(f,f_0) > \epsilon_n} \cap
        \mathcal{F}_n\mid \mathbf y, \mathbf{x})\mid \mathbf x\right] \right]\\
    &\quad+ Q_0^n\left[P_0^n\left[\big(1 - \phi_n(\mathbf y \mid \mathbf x)\big)
        \Pi_{\mathbf x}(\Set{f \given d_n(f,f_0) > \epsilon_n} \cap
        \mathcal{F}_n\mid \mathbf y, \mathbf x)\mid \mathbf x \right] \right].
  \end{aligned}
\end{multline}
Now, to any $\mathbf x\in \mathbb R^n$, we associate the event
\begin{equation}
  \label{eq:21}
  E_n(\mathbf x) := \Set*{y \in \mathbb R^n \given
    \int_{\mathcal{F}}\prod_{i=1}^n
    \frac{p_{f}(x_i,y_i)}{p_{f_0}(x_i,y_i)}\,d\Pi_{\mathbf x}(f) \geq
    \exp\left(-(1+ 4c_2)\frac{n\epsilon_n^2}{4}\right)}.
\end{equation}
Consider the first term of the rhs of \cref{eq:1}. We can rewrite,
\begin{multline*}
  Q_0^n[P_0^n[\Pi_{\mathbf x}(\mathcal{F}_n^c\mid \mathbf y, \mathbf x) \mid
  \mathbf x ]]\\
  \begin{aligned}
    &\leq e^{\frac 14(4c_2+1) n\epsilon_n^2}\int_{\mathbb R^n}\int_{E_n(\mathbf
      x)} \int_{\mathcal{F}_n^c} \prod_{i=1}^n
    \frac{p_{f}(x_i,y_i)}{p_{f_0}(x_i,y_i)}\, d\Pi_{\mathbf x}(f) dP_0^n(\mathbf
    y\mid \mathbf x)dQ_0^n(\mathbf x) \\
    &\quad+ \int_{\mathbb R^n}\int_{E_n(\mathbf
      x)^c}dP_0^n(\mathbf y\mid \mathbf x)dQ_0^n(\mathbf
    x)\\
    &=e^{\frac 14(4c_2+1) n\epsilon_n^2}
    \int_{\mathbb R^n} \int_{\mathcal{F}_n^c}
    \int_{E_n(\mathbf x)} dP^n(\mathbf y\mid \mathbf x) d\Pi_{\mathbf
      x}(f)dQ_0^n(\mathbf x)\\
    &\quad+ \int_{\mathbb R^n} \int_{E_n(\mathbf x)^c}
    dP_0^n(\mathbf y\mid \mathbf x)dQ_0^n(\mathbf x)\\
    &\leq e^{\frac 14(4c_2+1)n\epsilon_n^2}\int_{\mathbb R^n}\Pi_{\mathbf
      x}(\mathcal{F}_n^c)\,dQ_0^n(\mathbf x) + \int_{\mathbb
      R^n}\int_{E_n(\mathbf x)^c}dP_0^n(\mathbf y\mid \mathbf x)dQ_0^n(\mathbf
    x),
  \end{aligned}
\end{multline*}
where the third line follows from Fubini's theorem. The same reasoning applies
to the other terms of \cref{eq:1}, using the test functions introduced
above and $0 < c_2 < 1/4$. Hence the theorem is proved if we show that
$\int_{\mathbb R^n}\int_{E_n(\mathbf x)^c} dP_0^n(\mathbf y\mid \mathbf
x)dQ_0^n(\mathbf x) = o(1)$. But under the condition of the theorem,
\citet[Lemma~10]{GhosalVanDerVaartothers2007} implies that
\begin{equation*}
  P_0^n\left(
    \int_{\mathcal{F}}\prod_{i=1}^n\frac{p_f(x_i,Y_i)}{p_{f_0}(x_i,Y_i)}\,
    d\Pi_{\mathbf x}(f) < \exp\left(- \frac 14(1+4c_2)\epsilon_n^2\right)\
    \Bigg|\ \mathbf x \right) = o(1). \qedhere
\end{equation*}

\appendix

\section{Proofs of \cref{lem:4,lem:1} and some technical results on the kernels}
\label{sec:eta}

\subsection{Proof of \cref{lem:4}}

Clearly, $\|\chi_{\sigma} * f\|_1 \leq \|\chi_{\sigma}\|_1\|f\|_1$ by Young's
inequality, so that $\chi_{\sigma}*f \in L^1$ and
$(\chi_{\sigma}*f)^{\wedge}(\xi) =
\widehat{\chi}_{\sigma}(\xi)\widehat{f}(\xi)$, showing that the support of the
Fourier transform of $\chi_{\sigma}*f$ is included in
$[-1/\sigma,1/\sigma]$. Moreover, using again Young's inequality we get that
$\|\chi_{\sigma} * f\|_{\infty} \leq \|\chi_{\sigma}\|_1\|f\|_{\infty}$, thus
$\chi_{\sigma}*f \in L^{\infty}$.

\par Because $\widehat{\chi}$ is $\mathtt C^{\infty}$ and compactly supported,
for any integer $q \geq 0$ we have
$(i u)^q \chi (u) = (2\pi)^{-1}\int \widehat{\chi}^{(q)}(\xi) e^{i\xi
  u}d\xi$. Clearly $\widehat{\chi}$ is Schwartz, hence by Fourier inversion we
have that
\begin{align*}
  \int u^q \chi(u) e^{-i\xi u}du = (-i)^q \widehat{\chi}^{(q)}(\xi),\quad
  \forall \xi \in \mathbb R.
\end{align*}
But, by construction $\widehat{\chi}(0) = 1$, and for any $q\geq 1$ we have
$\widehat{\chi}^{(q)}(0) = 0$. It follows that $\int \chi (u) du = 1$, and
$\int u^q \chi(u) du = 0$ for any $q\geq 1$. Whence, letting $m$ be the largest
integer smaller than $\beta$, and using Taylor's formula with exact remainder
term
\begin{multline*}
  \chi_{\sigma}*f (x) - f(x)\\
  \begin{aligned}
    &= \int \chi_{\sigma}(y)\left[f(x - y) - f(x)\right]\, dy
    = \int \chi(y)\left[f(x - \sigma y) - f(x)\right]\, dy\\
    &= \sum_{k=1}^{m}\frac{(-1)^k \sigma^k}{k!} \int u^k \chi(u)\, du\\
    &\quad + \int \chi(y) \int_{0}^1 (-\sigma y)^m \frac{(1 - u)^{m-1}}{(m-1)!}
    \left[f^{(m)}(x - u\sigma y) - f^{(m)}(x)\right]\, dudy\\
    &= \int \chi(y) \int_{0}^1 (-\sigma y)^m \frac{(1 - u)^{m-1}}{(m-1)!}
    \left[f^{(m)}(x - u\sigma y) - f^{(m)}(x)\right]\, dudy.
  \end{aligned}
\end{multline*}
Therefore, because $f \in \mathtt C^{\beta}$,
\begin{multline*}
  |\chi_{\sigma}* f(x) - f(x)|\\
  \begin{aligned}
    &\leq
    \sigma^m \int |y^m\chi(y)|\int_0^1 \frac{(1 - u)^{m-1}}{(m-1)!}
    |f^{(m)}(x - u\sigma y) - f^{(m)}(x)|\, du dy\\
  &\leq \|f\|_{\mathtt C^{\beta}}\sigma^{\beta} \int |y^{\beta} \chi(y)|\, dy
    \int_0^1 \frac{(1 - u)^{m-1}}{(m-1)!} u^{\beta - m}\, du.
  \end{aligned}
\end{multline*}

\subsection{Proof of \cref{lem:1}}

We mostly follow the proof of \citet[proposition~1]{Hangelbroek2010}. Writing,
\begin{align*}
  K_{h,\sigma}f_{\sigma}(x)
  &= \int \frac{h}{\sigma} \sum_{k\in\mathbb Z} \varphi\left( \frac{x-h\sigma
    k}{\sigma }\right)
    \eta \left(\frac{y-h\sigma k}{\sigma} \right)f_{\sigma}(y)\,dy\\
  &= \frac{h}{\sigma} \sum_{k\in\mathbb Z} \varphi\left( \frac{x-h\sigma
    k}{\sigma} \right)
    \int \eta \left(\frac{y-h \sigma k}{\sigma} \right)f_{\sigma}(y)\,dy\\
  &= \frac{h}{2\pi}\sum_{k\in\mathbb Z} \varphi\left( \frac{x-h\sigma
    k}{\sigma} \right)
    \int \widehat{\eta}(\sigma\xi) \widehat{f}_{\sigma}(\xi) e^{i\xi h\sigma
    k}\,d\xi\\
  &=  \int \widehat{\eta}(\sigma \xi) \widehat{f}_{\sigma}(\xi)
    \frac{h}{2\pi}\sum_{k\in\mathbb Z} \varphi\left( \frac{x-h\sigma k}{\sigma
    } \right) e^{i\xi h \sigma k}\, d\xi.
\end{align*}
Then we can invoke the \textit{Poisson summation formula}
\citep[theorem~4.1]{HaerdleKerkyacharianPicardEtAl1998}, which is obviously
valid for $\varphi$, and
\begin{align*}
  \sum_{k\in \mathbb Z} \varphi\left(\frac{x - h \sigma k}{\sigma}\right)
  e^{i\xi h \sigma k}
  = \frac 1h \sum_{m\in \mathbb Z}\widehat{\varphi}\left(\sigma \xi +
  \frac{2\pi m}{h}\right)e^{i(\sigma \xi + \frac{2\pi m}{h}) x / \sigma}.
\end{align*}
Therefore, recalling that $\widehat{f}_{\sigma}$ is supported on
$[-1/\sigma,1/\sigma]$ and $\widehat{\chi}$ equals $1$ on $[-1,1]$,
\begin{align*}
  K_{h,\sigma}f_{\sigma}(x)
  &=
    \frac{1}{2\pi}\int \widehat{\chi}(\sigma \xi) \widehat{f}_\sigma(\xi)
    \sum_{m\in \mathbb Z} \frac{\widehat{\varphi}(\sigma\xi + 2\pi
    m /h)}{\widehat{\varphi}(\sigma\xi)} e^{i(\sigma\xi + \frac{2\pi m}{h}) x
    / \sigma }\, d\xi\\
  &=f_{\sigma}(x) + \frac{1}{2\pi}\sum_{m\in \mathbb{Z}\backslash \Set{0}}
    \int \widehat{f}_{\sigma}(\xi) \frac{\widehat{\varphi}(\sigma\xi + 2\pi
    m/h)}{\widehat{\varphi}(\sigma\xi)} e^{i(\sigma\xi + \frac{2\pi m}{h}) x /
    \sigma }\, d\xi.
\end{align*}
It follows that,
\begin{align*}
  |K_{h,\sigma}f_{\sigma}(x) - f_{\sigma}(x)|
  &\leq \frac{1}{2\pi}\|\widehat{f}_{\sigma}\|_1 \sup_{\xi \in
    [-1,1]}\sum_{m\in \mathbb Z\backslash \Set{0}}\left |
    \frac{\widehat{\varphi}(\xi +  2\pi m /h )}{\widehat{\varphi}(\xi)}
    \right|.
\end{align*}
Now,
$\|\widehat{f}_{\sigma}\|_1 \leq 2\sigma^{-1}\|\widehat{f}_{\sigma}\|_{\infty}
\leq 2 \sigma^{-1}\|f_{\sigma}\|_1 \leq 2\sigma^{-1}\|f\|_1$, which is finite
because of \cref{lem:4}. Recalling that by assumption $\widehat{\varphi}$ is
Gaussian, it follows for all $\xi \in [-1,1]$ and all $h \leq 1$,
\begin{align*}
  \sum_{m\in \mathbb Z\backslash \Set{0}} \left | \frac{\widehat{\varphi}(\xi
  + 2\pi m  /h )}{\widehat{\varphi}(\xi)} \right|
  &\leq \exp\left\{
    -\frac 12 (\xi + 2\pi m /h)^2 + \frac 12 \xi^2 \right\}\\
  &\leq e^{-1/2}
    \sum_{m\in \mathbb Z \backslash \Set{0}} e^{-4\pi^2m^2/h^2}
    \leq
    4e^{-1/2}e^{-4\pi^2/h^2}.
\end{align*}
Then the lemma is proved with $C := 8e^{-1/2}$.

\subsection{Some other technical results on $K_{h,\sigma}$}

\begin{lemma}
  \label{pro:4}
  There is a universal constant $C > 0$ such that for all $x\in \mathbb R$, all
  $0 < h \leq 1$ and all $\sigma > 0$,
  $\sum_{k\in\mathbb Z}|\eta((x - h\sigma k)/\sigma)| \leq C h^{-1}$. Moreover,
  $\eta \in \mathcal{S}$.
\end{lemma}
\begin{proof}
  \par We first prove that $\eta \in \mathcal{S}$. Obviously
  $\widehat{\varphi} \in \mathcal{S}$, and therefore so is
  $\widehat{\eta}$. Since the Fourier transform and the inverse Fourier
  transform are continuous mapping of $\mathcal{S}$ onto itself, it is immediate
  that $\eta \in \mathcal{S}$.

  \par We finish the proof by remarking that
  $x \mapsto \sum_{k\in\mathbb Z}|\eta((x - h\sigma k)/\sigma)|$ is periodic
  with period $h\sigma$, hence it suffices to check that it is bounded for
  $x \in [0,h\sigma]$. If $x\in [0,h\sigma]$, then
  $|x - h\sigma k| \geq |h \sigma k| / 2$ for any $|k| \geq 2$, so that 
  \begin{align*}
    \sum_{k\in \mathbb Z}|\eta((x - h\sigma k)/\sigma)|
    &\leq 3 \sup_{u\in \mathbb R}|\eta(u)| + \sum_{|k|\geq 2}|\eta((x -
      h\sigma k)/\sigma)|\\
    &\leq 3 \|\eta\|_{0,0} +\|\eta\|_{2,0} \sum_{|k|\geq 2}\left(1 + |hk|/2
      \right)^{-2}\\
    &\leq 3\|\eta\|_{0,0} + 4\|\eta\|_{2,0} / h,
  \end{align*}
  which concludes the proof of the first assertion with
  $C :=3\|\eta\|_{0,0} + 4\|\eta\|_{2,0} $, because of the assumption
  $h \leq 1$.
\end{proof}

The following Lemma gives some control on the coefficients of $f$ on $\eta$.

\begin{proposition}
  \label{pro:2}
  Let $0 < h \leq 1$ and
  $a_k(f) := (h/\sigma)\int \eta((y-h\sigma k)/\sigma)f(y)\,dy$x. Then there are
  universal constants $C,C' > 0$, depending only on $\varphi$, such that
  $\sum_{k\in \mathbb Z}|a_k(f)| \leq C \|f\|_1 \sigma^{-1}$, and for all
  $k \in \mathbb Z$, $|a_k(f)| \leq C' \|f\|_{\infty}$.
\end{proposition}
\begin{proof}
  For the first assertion of the proposition, we write,
  \begin{align*}
    \sum_{k\in \mathbb Z}|a_k(f)|
    &\leq \frac{h}{\sigma}\sum_{k\in\mathbb Z}
      \int |f(y)| |\eta((y-h\sigma k)/\sigma)|\, dy\\
    &\leq \sigma^{-1}\|f\|_1 \sup_{y\in \mathbb R} h \sum_{k\in\mathbb
      Z}|\eta((y-h\sigma k)/\sigma)|,
  \end{align*}
  and the conclusion follows from \cref{pro:4}. The proof of the second
  assertion is simpler. Indeed,
  \begin{align*}
    |a_k(f)|
    &\leq \frac{h}{\sigma} \int
      |f(y)| |\eta((y-h\sigma k)/\sigma)|\, dy
      \leq h \|f\|_{\infty}\int |\eta(u)|\, du,
  \end{align*}
  where the last integral is bounded because $\eta \in \mathcal{S}$ by
  \cref{pro:4}.
\end{proof}

\section{Proof of \cref{lem:8}}
\label{sec:prlem8}

Let $x \in \mathbb R^n$ arbitrary, $\sigma > 0$ and
$h_{\sigma}\sqrt{\log \sigma^{-1}} := 2\pi\sqrt{\beta +1}$. Recall that from
\cref{lem:4,lem:1} we have
$\|K_{h_{\sigma},\sigma}(\chi_{\sigma}*f_0) - f_0\|_{\infty} \lesssim
\sigma^{\beta}$, where
$K_{h_{\sigma},\sigma}(\chi_{\sigma}*f_0)(z) := \sum_{k\in \mathbb Z}u_k\,
\varphi((z - h_{\sigma}\sigma k)/\sigma)$. Define
$S_n(x) := \cup_{i=1}^n\Set{z \in \mathbb R \given |z - x_i| \leq
  \sigma\sqrt{2(\beta + 1)\log \sigma^{-1}}}$ and
\begin{equation*}
  \Lambda(x) := \Set{ k \in \mathbb Z \given |u_k| > \sigma^{\beta},\quad
    h_{\sigma} \sigma k \in S_n(x) }.
\end{equation*}
Also define
$U_{\sigma} := \Set{\sigma' \given \sigma \leq \sigma' \leq \sigma(1+
  \sigma^{\beta)}}$, and for all $k \in \Lambda(x)$ define
$V_k := \Set{\mu \given |\mu - h_{\sigma}\sigma k| \leq \sigma^{\beta+1}}$. We
denote by $\mathcal{M}_{\sigma}$ the set of signed measures $M$ on $\mathbb R$
such that $|M(V_k) - u_k| \leq \sigma^{\beta}$ for all $k\in \Lambda(x)$ and
$|M|(V^c) \leq \sigma^{\beta}$, where $V^c$ is the relative complement of the
union of all $V_k$ for $k\in \Lambda(x)$. For any $M\in \mathcal{M}_{\sigma}$,
we write $f_{M,\sigma}(z) := \int \varphi((z-\mu)/\sigma)\,dM(\mu)$. Act as in
\cref{pro:10} to find that $d_n(f,f_0) \leq C h_{\sigma}^{-2}\sigma^{\beta}$ for
any $M \in \mathcal{M}_{\sigma}$, with a constant $C > 0$ not depending on
$x$. By construction of $S_n(x)$, for all $k\in \Lambda(x)$ there is at least
one $x_i$ such that
$|h_{\sigma}\sigma k - x_i| \leq \sigma\sqrt{2(\beta+1)\log \sigma^{-1}}$. Then
for any $k\in \Lambda(x)$, by definition of $G_x$
\begin{equation*}
  \overline{\alpha}G_x(V_k)
  \geq n^{-1}\int_{h_{\sigma}\sigma k - \sigma^{\beta+1}}^{h_{\sigma}\sigma k
    + \sigma^{\beta +1}}g(z - x_i)\, dz
  \geq a_{21}n^{-1}\sigma^{a_{22}(\beta + 1)}.
\end{equation*}
Remarking that $|\Lambda(x)| \lesssim \sigma^{-(\beta+1)}$ independently of $x$
(see \cref{pro:13}) and letting $\epsilon_n = C'h_{\sigma}^{-2}\sigma^{\beta}$
we can mimic the steps of \cref{sec:kullb-leibl-cond-2} to find that
\begin{align*}
  \Pi_x(f\, :\, d_n(f,f_0) \leq s\epsilon)
  &\gtrsim
  \exp\left\{-C''|\Lambda(x)|\log \sigma^{-1} - C''|\Lambda(x)|\log n \right\}\\
  &\gtrsim \exp(-c_2 n\epsilon_n^2),
\end{align*}
for a constant $C''>0$ not depending on $x$ and $\epsilon_n^2$ defined in the
lemma.

\section{Some technical results on the construction of the approximation in the
  case of location-scale mixtures}
\label{pr:locationscale}

\begin{proposition}
  \label{pro:5}
  Let $f_0\in \mathcal C_\beta$. For any $j\geq 0$, we have
  $|\Delta_j(x)| \leq C \|f_0\|_{\mathtt C^{\beta}} \sigma_j^{\beta}$, with the
  same constant $C> 0$ as in \cref{lem:4}. Moreover,
  $\|\Delta_j\|_1 \leq 2\|f_0\|_1$ for all $j\geq 0$.
\end{proposition}
\begin{proof}
  Notice that
  $\|\Delta_{j+1}\|_1 \leq \|\Delta_j\|_1 + \|\chi_{\sigma_{j+1}}*\Delta_j\|
  \leq (1 + \|\chi\|_1) \|\Delta_j\|_1$, by Young's inequality. Since
  $f_0\in L^1$, this implies $\Delta_j \in L^1$ for all $j\geq 0$. Since
  $\widehat{\Delta}_{j+1}(\xi) = \widehat{\Delta}_j(\xi) -
  \widehat{\chi}_{\sigma_{j+1}}(\xi)\widehat{\Delta}_j(\xi)$, we get
  $\widehat{\Delta}_j(\xi) = \widehat{f_0}(\xi)\prod_{l=1}^j\left(1 -
    \chi_{\sigma_l}(\xi) \right)$, by induction. Because
  $\sigma_{j+1} = \sigma_j / 2$, and by construction of $\chi_{\sigma_l}$ we
  have
  $\widehat{\chi}_{\sigma_m}(\xi)\widehat{\chi}_{\sigma_l}(\xi) =
  \widehat{\chi}_{\sigma_m}(\xi)$ for any $m > l$, hence the last equation can
  be rewritten as
  $\widehat{\Delta}_j(\xi) = \widehat{f}_0(\xi)(1 -
  \widehat{\chi}_{\sigma_j}(\xi))$. Then we deduce that
  $\Delta_j = f_0 - \chi_{\sigma_j}* f_0$. By \cref{lem:4}, this implies that
  $|\Delta_j(x)| \leq C\|f_0\|_{\mathtt C^{\beta}} \sigma_j^{\beta}$. From the
  same estimate, it is clear that
  $\|\Delta_j\| \leq \|f_0|_1 + \|\chi_{\sigma_j}*f_0\| \leq 2\|f_0\|_1$.
\end{proof}

\subsection{Proof of \cref{pro:3}}
\label{sec:prpro3}

Let define $A(\beta,J) := (2\log|\Lambda| + 2\beta \log\sigma_J^{-1})^{1/2}$ and
$\mathcal{J} \equiv \mathcal{J}(x) := \Set{(j,k)\in \Set{0,\dots,J}\times
  \mathbb Z \given |x - \mu_{jk}| \leq 4 A(\beta,J)\sigma_j}$. For any
$M \in \mathcal{M}$ we can write
\begin{multline*}
  f_M(x) - f_0(x) = \sum_{(j,k) \in \Lambda \cap \mathcal{J}} \int_{W_{jk}}
  \left[ \varphi\left( \frac{x - \mu}{\sigma} \right) - \varphi\left( \frac{x -
        \mu_{jk}}{\sigma_j} \right) \right]\, dM(\sigma,\mu)\\
  + \sum_{(j,k)\in \Lambda \cap \mathcal{J}}\left[M(W_{jk}) - u_{jk} \right]
  \varphi\left( \frac{x - \mu_{jk}}{\sigma_j} \right) + \sum_{(j,k) \in \Lambda
    \cap \mathcal{J}^c}\int_{W_{jk}} \varphi\left(
    \frac{x - \mu}{\sigma} \right)\, dM(\sigma,\mu)\\
  - \sum_{(j,k)\in \Lambda \cap \mathcal{J}^c} u_{jk}\, \varphi\left( \frac{x -
      \mu_{jk}}{\sigma_j} \right) - \sum_{(j,k) \notin \Lambda}u_{jk}\,
  \varphi\left( \frac{x - \mu_{jk}}{\sigma_j} \right)\\
  + \int_{W^c}\varphi\left(\frac{x -
      \mu}{\sigma} \right)\, dM(\sigma,\mu)\\
  := r_1(x) + r_2(x) + r_3(x) + r_4(x) + r_5(x) + r_6(x).
\end{multline*}
The proof follows similar steps as the proof of \cref{pro:10}. From the
definition of $A(\beta,J)$ and \cref{pro:6}, we deduce that
$A(\beta, J) \lesssim \sqrt{J}$ for $J$ large enough. Also, there is a
separation of $h_J\sigma_j$ between two consecutive $\mu_{jk}$. Then there are
no more than $2A(\beta,J)\sigma_j/(h_J\sigma_j) = 2A(\beta,J)h_J^{-1}$ distinct
values of $\mu_{jk}$ in an interval of length $2A(\beta,J)\sigma_j$. Thus the
bound $|\Lambda \cap \mathcal{J}| \leq 2(J+1)A(\beta, J) \lesssim J^{3/2}$
holds. It follows from \cref{pro:1} that
$|r_1(x)| \lesssim |\Lambda \cap \mathcal{J}| \sigma_J^{\beta} \lesssim J^{3/2}
\sigma_J^{\beta}$. Obviously,
$|r_2(x)| \leq \|\varphi\|_{\infty} |\Lambda \cap \mathcal{J}| \sigma_J^{\beta}
\lesssim J^{3/2} \sigma_J^{\beta}$. Whenever
$(j,k) \in \Lambda \cap \mathcal{J}^c$ and $(\sigma,\mu) \in W_{jk}$, choosing
$J$ large enough so that $1/2\leq \sigma_j/\sigma \leq 2$ and
$|\mu - \mu_{jk}| \leq \sigma_j A(\beta,J)/2$, it holds
$|x - \mu|\geq A(\beta,J)\sigma$. Therefore,
$|r_3(x)| \lesssim \varphi(A(\beta,J)) |\Lambda| \leq \sigma_J^{\beta}$. With
the same reasoning we get $|r_4(x)| \lesssim \|f\|_{\infty}
\sigma_J^{\beta}$. Regarding $r_6$, we have the obvious bound
$|r_6(x)| \leq \|\varphi\|_{\infty}\sigma_J^{\beta}$. The $r_5$ term is more
subtle and constitutes the remainder of the proof.

Let
$\Lambda_1^c := \Set{(j,k)\in \Set{0,\dots,J}\times \mathbb Z \given |u_{jk}|
  \leq \sigma_J^{\beta}}$ and
$\mathcal{K}_j := \Set{k \in \mathbb Z \given |\mu_{jk}| > \zeta_j +
  \sqrt{2(\beta + 1)\log \sigma_J^{-1}} }$. Assuming that
$x \in [-\zeta_q,\zeta_q]$ for some $0 \leq q \leq J$, we can bound $r_5(x)$ as
follows,
\begin{multline}
  \label{eq:7}
  |r_5(x)| \leq \sum_{(j,k)\in \Lambda_1^c}|u_{jk}|\, \varphi\left( \frac{x -
      \mu_{jk}}{\sigma_j} \right)\\ + \sum_{j \leq q}\sum_{k\in \mathcal{K}_j}
  u_{jk}\, \varphi\left( \frac{x - \mu_{jk}}{\sigma_j} \right) + \sum_{j >
    q}\sum_{k\in \mathcal{K}_j} u_{jk}\, \varphi\left( \frac{x -
      \mu_{jk}}{\sigma_j} \right),
\end{multline}
where the third term of the rhs does not exist if $q=J$. The first term of the
rhs of \cref{eq:7} is bounded by
$\sigma_J^{\beta}\sup_{x\in \mathbb R} \sum_{j=0}^J \sum_{k\in \mathbb Z}
\varphi((x-\mu_{jk})/\sigma)$, which is in turn bounded by a constant multiple
of $J^{3/2}\sigma_J^{\beta}$ (see for instance the proof of
\cref{pro:4}). Because of \cref{pro:2,pro:5}, when $x \in [-\zeta_q,\zeta_q]$ we
always have
\begin{align*}
  \sum_{j\leq q}\sum_{k\in \mathcal{K}_j} |u_{jk}|\, \varphi\left(
  \frac{x - \mu_{jk}}{\sigma_j} \right)
  &\leq \sup_{\substack{j\leq q\\ k\in \mathcal{K}_j}}
  \varphi\left(
  \frac{x - \mu_{jk}}{\sigma_j} \right)
  \sum_{j\leq J}\sum_{k\in \mathbb Z}|u_{jk}|\\
  &\leq \sigma_J^{\beta+1} \sum_{j\leq J} 2\|f_0\|_1\sigma_j^{-1}
    \leq 4\|f_0\|_1 \sigma_J^{\beta}.
\end{align*}
Regarding the second term of the rhs of \cref{eq:7}, we introduce the sets of
indexes
$\mathcal{L}_j \equiv \mathcal{L}_j(x) := \Set{ k\in \mathcal{K}_j \given |x -
  \mu_{jk}| \leq \sigma_j \sqrt{2(\beta + 1)\log \sigma_J^{-1}}}$. Then, we can
split again the sum as
\begin{multline*}
  \sum_{j > q}\sum_{k\in \mathcal{K}_j} u_{jk}\, \varphi\left( \frac{x -
      \mu_{jk}}{\sigma_j} \right) =\\ \sum_{j>q} \sum_{k\notin
    \mathcal{L}_j}u_{jk}\, \varphi\left( \frac{x - \mu_{jk}}{\sigma_j} \right) +
  \sum_{j>q} \sum_{k \in \mathcal{L}_j}u_{jk}\, \varphi\left( \frac{x -
      \mu_{jk}}{\sigma_j} \right).
\end{multline*}
With exactly the same reasoning as before, we get that the first sum of the rhs
of the last expression is bounded above by
$4\|f\|_1\sigma_J^{\beta}$. Concerning the second term, for any $j\geq 1$ we get
from \cref{pro:2,pro:5}, together with the definition of $u_{jk}$, that
$|u_{jk}| \lesssim \|f\|_{\mathtt C^{\beta}}\sigma_j^{\beta}$. Since there is
$h_J\sigma_j$ separation between two consecutive $\mu_{jk}$, we deduce that
$|\mathcal{L}_j| \leq 2h_J^{-1}\sqrt{2(\beta + 1)\log
  \sigma_J^{-1}}$. Therefore, for $J$ large enough and
$x \in [-\zeta_q,\zeta_q]$ with $0 \leq q \leq J$,
\begin{equation*}
  |r_5(x)|
  \lesssim \|f_0\|_1 \sigma_J^{\beta}
  + \|f_0\|_{\mathtt C^{\beta}}\sqrt{2(\beta + 1)\log
    \sigma_J^{-1}} \sum_{j> q}\sigma_j^{\beta}
  \lesssim \sqrt{J}\sigma_q^{\beta}.
\end{equation*}
The conclusion of the proposition follows by combining all the preceding points.

\section{Elementary results}
\label{sec:elementary-results}

\begin{proposition}
  \label{pro:1}
  Let $\varphi(x) = \exp(-x^{2}/2)$. Then, for all $\mu_1,\mu_2 \in \mathbb R$,
  and all $\sigma_1,\sigma_2 > 0$ with $1/2 \leq \sigma_1/\sigma_2 \leq 2$,
  \begin{align*}
    \sup_{x\in\mathbb R}\left|\varphi\left(\frac{x - \mu_1}{\sigma_1}\right) -
    \varphi\left(\frac{x - \mu_2}{\sigma_2}\right)\right|
    &\leq
      4\frac{|\sigma_1 - \sigma_2|}{\sigma_1 \vee
      \sigma_2} + \frac{|\mu_1 - \mu_2|}{\sigma_1 \vee \sigma_2}.
  \end{align*}
\end{proposition}
\begin{proof}
  Without loss of generality we can assume that $\sigma_1 \leq \sigma_2$.  Using
  the triangle inequality, we first write
  \begin{multline}
    \left|\varphi\left(\frac{x - \mu_1}{\sigma_1}\right) -
      \varphi\left(\frac{x - \mu_2}{\sigma_2}\right)\right|\\
    \begin{aligned}
      &\leq \left|\varphi\left(\frac{x - \mu_1}{\sigma_1}\right) -
        \varphi\left(\frac{x - \mu_2}{\sigma_1}\right)\right| +
      \left|\varphi\left(\frac{x - \mu_2}{\sigma_1}\right) -
        \varphi\left(\frac{x - \mu_2}{\sigma_2}\right)\right|\\
      &\leq \sup_{u\in \mathbb R}\left|\varphi\left(u + \frac{\mu_1 -
            \mu_2}{\sigma_1} \right) - \varphi(u)\right| + \sup_{u\in \mathbb
        R}\left| \varphi\left(\frac{\sigma_1}{\sigma_2}u\right) - \varphi(u)
      \right|.
    \end{aligned}
    \label{eq:3}
  \end{multline}
  The first term of the rhs of \cref{eq:3} is obviously bounded by
  $|\mu_1 - \mu_2| / \sigma_1$. Regarding the second term of the rhs of
  \cref{eq:3},
  \begin{equation*}
    \left|
      \varphi\left(\frac{\sigma_1}{\sigma_2}u\right) - \varphi(u) \right| \leq
    |\sigma_1/\sigma_2 - 1| \left( \frac{\sigma_1}{\sigma_2} \vee 1\right)^2
    \sup_{x} x^2 \varphi( x),
  \end{equation*}
  which terminates the proof.
\end{proof}

\begin{proposition}
  \label{pro:7}
  Let $X \sim \mathrm{SGa}(\alpha,1)$, with $0 < \alpha \leq 1$. Then for
  any $x \in \mathbb R$ and any $0< \delta \leq 1/2$ we have
  $\Pr\Set{|X - x| \leq \delta} \geq \frac{\delta e^{-2|x|}}{3e\Gamma(\alpha)}$.
\end{proposition}
\begin{proof}
  Assume for instance that $x \geq 0$. Recalling that $X$ is distributed as the
  difference of two independent $\mathrm{Ga}(\alpha,1)$ distributed random
  variables, it follows
  \begin{align*}
    \Pr\Set{|X - x| \leq \delta}
    &\geq
      \frac{1}{\Gamma(\alpha)} \int_0^{\infty} y^{\alpha -
      1}e^{-y}
      \frac{1}{\Gamma(\alpha)}
      \int_{x+y}^{x+y + \delta} z^{\alpha-1}e^{-z}\, dz\, dy.
  \end{align*}
  Because $\alpha \leq 1$, the mapping $z \mapsto z^{\alpha-1}e^{-z}$ is
  monotonically decreasing on $\mathbb R^+$, then the last integral in the rhs
  of the previous equation is lower bounded by
  $\delta (x+y+\delta)^{\alpha -1}e^{-(x+y+\delta)} \geq \delta
  e^{-2(x+y+\delta)}$. Then
  \begin{align*}
    \Pr\Set{|X - x| \leq \delta}
    &\geq
      \frac{\delta e^{-2(x+\delta)}}{\Gamma(\alpha)^2} \int_0^{\infty}
      y^{\alpha - 1}e^{-3y}\, dy
      \geq \frac{3^{-\alpha}e^{-2(x+\delta)}}{\Gamma(\alpha)}\delta
      \geq \frac{\delta e^{-2|x|}}{3e\Gamma(\alpha)}.
  \end{align*}
  The proof when $x < 0$ is obvious.
\end{proof}

\bibliographystyle{abbrvnat}
\bibliography{bib/biblio-paper}

\begin{thebibliography}{22}
\providecommand{\natexlab}[1]{#1}
\providecommand{\url}[1]{\texttt{#1}}
\expandafter\ifx\csname urlstyle\endcsname\relax
  \providecommand{\doi}[1]{doi: #1}\else
  \providecommand{\doi}{doi: \begingroup \urlstyle{rm}\Url}\fi

\bibitem[Birg{\'e}(2006)]{Birge2006}
Lucien Birg{\'e}.
\newblock Model selection via testing: an alternative to (penalized) maximum
  likelihood estimators.
\newblock In \emph{Annales de l'IHP Probabilit{\'e}s et statistiques},
  volume~42, pages 273--325, 2006.

\bibitem[Bochkina and Rousseau(2016)]{bochkina:rousseau:16}
N~Bochkina and J~Rousseau.
\newblock {Adaptive density estimation based on a mixture of Gammas}.
\newblock \emph{ArXiv e-prints}, May 2016.

\bibitem[Canale and De~Blasi(2013)]{Canale2013}
Antonio Canale and Pierpaolo De~Blasi.
\newblock Posterior consistency of nonparametric location-scale mixtures for
  multivariate density estimation.
\newblock \emph{arXiv preprint arXiv:1306.2671}, 2013.

\bibitem[de~Jonge and van Zanten(2010)]{dejonge:vzanten:09}
R.~de~Jonge and {J. H.} van Zanten.
\newblock Adaptive nonparametric {B}ayesian inference using location-scale
  mixture priors.
\newblock \emph{Ann. Statist.}, 38:\penalty0 3300--3320, 2010.

\bibitem[Ferguson(1973)]{Ferguson1973}
Thomas~S Ferguson.
\newblock A bayesian analysis of some nonparametric problems.
\newblock \emph{The annals of statistics}, pages 209--230, 1973.

\bibitem[Ghosal and {van der Vaart}(2007)]{ghosal2007bis}
Subhashis Ghosal and Aad {van der Vaart}.
\newblock Convergence rates of posterior distributions for noniid observations.
\newblock \emph{Ann. Statist.}, 35\penalty0 (1):\penalty0 192--223, 02 2007.
\newblock \doi{10.1214/009053606000001172}.
\newblock URL \url{http://dx.doi.org/10.1214/009053606000001172}.

\bibitem[Ghosal and Van Der~Vaart(2001)]{Ghosal2001}
Subhashis Ghosal and Aad~W Van Der~Vaart.
\newblock Entropies and rates of convergence for maximum likelihood and bayes
  estimation for mixtures of normal densities.
\newblock \emph{Annals of Statistics}, pages 1233--1263, 2001.

\bibitem[Ghosal et~al.(2000)Ghosal, Ghosh, and {Van Der Vaart}]{Ghosal2000}
Subhashis Ghosal, {Jayanta K} Ghosh, and {Aad W} {Van Der Vaart}.
\newblock Convergence rates of posterior distributions.
\newblock \emph{Annals of Statistics}, 28\penalty0 (2):\penalty0 500--531,
  2000.

\bibitem[Ghosal et~al.(2007{\natexlab{a}})Ghosal, Van Der~Vaart,
  et~al.]{Ghosal2007}
Subhashis Ghosal, Aad Van Der~Vaart, et~al.
\newblock Posterior convergence rates of dirichlet mixtures at smooth
  densities.
\newblock \emph{The Annals of Statistics}, 35\penalty0 (2):\penalty0 697--723,
  2007{\natexlab{a}}.

\bibitem[Ghosal et~al.(2007{\natexlab{b}})Ghosal, Van Der~Vaart,
  et~al.]{GhosalVanDerVaartothers2007}
Subhashis Ghosal, Aad Van Der~Vaart, et~al.
\newblock Convergence rates of posterior distributions for noniid observations.
\newblock \emph{The Annals of Statistics}, 35\penalty0 (1):\penalty0 192--223,
  2007{\natexlab{b}}.

\bibitem[Goldenshluger and Lepski(2014)]{Goldenshluger2014}
A.~Goldenshluger and O.~Lepski.
\newblock On adaptive minimax density estimation on $r^d$.
\newblock \emph{Probability Theory and Related Fields}, 159\penalty0
  (3):\penalty0 479--543, 2014.
\newblock ISSN 1432-2064.
\newblock \doi{10.1007/s00440-013-0512-1}.
\newblock URL \url{http://dx.doi.org/10.1007/s00440-013-0512-1}.

\bibitem[Hangelbroek and Ron(2010)]{Hangelbroek2010}
Thomas Hangelbroek and Amos Ron.
\newblock Nonlinear approximation using gaussian kernels.
\newblock \emph{Journal of Functional Analysis}, 259\penalty0 (1):\penalty0
  203--219, 2010.

\bibitem[H{\"a}rdle et~al.(1998)H{\"a}rdle, Kerkyacharian, Picard, and
  Tsybakov]{HaerdleKerkyacharianPicardEtAl1998}
Wolfgang H{\"a}rdle, Gerard Kerkyacharian, Dominique Picard, and Alexander
  Tsybakov.
\newblock Wavelets.
\newblock In \emph{Wavelets, Approximation, and Statistical Applications},
  pages 1--16. Springer, 1998.

\bibitem[Hjort et~al.(2010)Hjort, Holmes, M{\"u}ller, and
  Walker]{hjort:holmes:mueller:2010}
N.~L. Hjort, C.~Holmes, P.~M{\"u}ller, and S.~G. Walker.
\newblock \emph{Bayesian Nonparametrics}.
\newblock Cambridge University Press, Cambridge, UK, 2010.

\bibitem[Kingman(1992)]{Kingman1992}
John Frank~Charles Kingman.
\newblock \emph{Poisson processes}, volume~3.
\newblock Oxford university press, 1992.

\bibitem[Kruijer et~al.(2010)Kruijer, Rousseau, and {van der
  Vaart}]{kruijer:rousseau:vdv:10}
W.~Kruijer, J.~Rousseau, and A.~{van der Vaart}.
\newblock {Adaptive {B}ayesian density estimation with location-scale
  mixtures}.
\newblock \emph{Electron. J. Stat.}, 4:\penalty0 {1225--1257}, 2010.

\bibitem[Naulet and Barat(2015)]{NauletBarat2015}
Zacharie Naulet and Eric Barat.
\newblock Some aspects of symmetric gamma process mixtures.
\newblock \emph{arXiv preprint arXiv:1504.00476}, 2015.

\bibitem[Reynaud-Bouret et~al.(2011)Reynaud-Bouret, Rivoirard, and
  Tuleau-Malot]{reynaud:rivoirard:tuleau11}
P.~Reynaud-Bouret, V.~Rivoirard, and C.~Tuleau-Malot.
\newblock Adaptive density estimation: a curse of support?
\newblock \emph{Journal of Statistical Planning and Inference}, 141:\penalty0
  115--139, 2011.

\bibitem[Salomond(2013)]{Salomond2013}
Jean-Bernard Salomond.
\newblock Bayesian testing for embedded hypotheses with application to shape
  constrains.
\newblock \emph{arXiv preprint arXiv:1303.6466}, 2013.

\bibitem[Scricciolo(2014)]{scricciolo:12}
C.~Scricciolo.
\newblock Adaptive {B}ayesian density estimation in ${L}^p$-metrics with
  {P}itman-{Y}or or normalized inverse-{G}aussian process kernel mixtures.
\newblock \emph{Bayesian Analysis}, 9:\penalty0 475--520, 2014.

\bibitem[Shen et~al.(2013)Shen, Tokdar, and Ghosal]{ShenTokdarGhosal2013}
Weining Shen, Surya~T Tokdar, and Subhashis Ghosal.
\newblock Adaptive bayesian multivariate density estimation with dirichlet
  mixtures.
\newblock \emph{Biometrika}, 100\penalty0 (3):\penalty0 623--640, 2013.

\bibitem[Wolpert et~al.(2011)Wolpert, Clyde, and Tu]{WolpertClydeTu2011}
Robert~L Wolpert, Merlise~A Clyde, and Chong Tu.
\newblock Stochastic expansions using continuous dictionaries: L{\'e}vy
  adaptive regression kernels.
\newblock \emph{The Annals of Statistics}, pages 1916--1962, 2011.

\end{thebibliography}

\end{document}